\documentclass[11pt]{amsart}
\usepackage[top=3.5cm, bottom=3.5cm, left=2.5cm, right=2.5cm]{geometry}
\usepackage[numbers]{natbib}
\usepackage{hyperref}

\title{Uniformly Almost Flatness and Solubility in Finitely Generated Groups
}

\author{David Guo}
\usepackage{graphicx}
\usepackage{framed}
\usepackage{amsmath}
\usepackage{amssymb}
\usepackage{amsfonts}
\usepackage{mathrsfs}
\usepackage{array}
\usepackage{enumerate}
\usepackage{comment}
\usepackage{amsthm}  
\usepackage{mathtools}
\usepackage{physics}
\usepackage[dvipsnames]{xcolor}
\usepackage{braket}
\usepackage{MnSymbol}

\usepackage{dutchcal}

 \usepackage{tikz}
 \usetikzlibrary{matrix}

\usepackage{soul}

\DeclareMathOperator{\lcm}{lcm}
\DeclareMathOperator{\diam}{diam}

\DeclareMathOperator{\GL}{GL}

\DeclareMathOperator{\ord}{ord}

\usepackage[capitalise,nameinlink ,noabbrev]{cleveref}

\providecommand{\C}{\mathbb{C}}

\renewcommand{\ge}{\mathcal{G}}
\providecommand{\F}{\mathbb{F}}
\providecommand{\Fp}{\mathbb{F}_p}

\providecommand{\N}{\mathbb{N}}
\providecommand{\Z}{\mathbb{Z}}
\providecommand{\Q}{\mathbb{Q}}

\renewcommand{\a}{\alpha}

\providecommand{\iso}{ \cong}

\renewcommand{\l}{ \lambda}

\providecommand{\d}{ \delta}

\renewcommand{\Pr}{ \mathcal{P}}

\providecommand{\qx}{\Q[x]}

\renewcommand{\i}{^{-1}}

\providecommand{\d}[2]{\diam(G/H)} 
\providecommand{\F}[1]{\Fit{#1}}

\providecommand{\Pi}{P_{i}}

\providecommand{\qx}{\Q[x]}

\theoremstyle{plain} 
\newtheorem{thm}{Theorem}[section] 

\definecolor{softgreen}{RGB}{0,100,0} 

\newtheorem*{thm*}{Theorem}
\newtheorem{lem}[thm]{Lemma} 
\newtheorem{lemma}[thm]{Lemma} 
 
\newtheorem{prop}[thm]{Proposition} 
\newtheorem{cor}[thm]{Corollary}

\usepackage{MnSymbol,wasysym}

\newcommand*{\myproofname}{Proof of the claim}

\theoremstyle{definition} 
\newtheorem{defn}[thm]{Definition}

\theoremstyle{remark} 
\newtheorem*{rem}{Remark} 

\makeatletter\@addtoreset{case}{thm}\makeatother

\DeclarePairedDelimiterX{\inp}[2]{\langle}{\rangle}{#1, #2}

\usepackage[numbers]{natbib}
\usepackage{cite}
\usepackage{url}

\newcommand{\david}[1]{ {\color{blue} David: } {\color{blue!75!black} #1} }

 \renewcommand{\david}[1]{}
\newtheorem{conjecture}[thm]{Conjecture}

\usepackage{soul}
\usepackage{nicematrix}

 \usepackage{tikz}
 \usetikzlibrary{matrix}

\providecommand{\C}{\mathbb{C}}

\renewcommand{\ge}{\mathcal{G}}
\providecommand{\F}{\mathbb{F}}
\providecommand{\Fp}{\mathbb{F}_p}

\providecommand{\N}{\mathbb{N}}
\providecommand{\Z}{\mathbb{Z}}
\providecommand{\Q}{\mathbb{Q}}

\renewcommand{\a}{\alpha}

\providecommand{\iso}{ \cong}

\renewcommand{\l}{ \lambda}

\providecommand{\d}{ \delta}

\renewcommand{\Pr}{ \mathcal{P}}

\renewcommand{\i}{^{-1}}
 
\providecommand{\d}[2]{\diam(G/H)} 
\providecommand{\F}[1]{\Fit{#1}}

\providecommand{\Pi}{P_{i}}

\providecommand{\d}{ \delta}

\renewcommand{\Pr}{ \mathcal{P}}

\providecommand{\qx}{\Q[x]}

\renewcommand{\i}{^{-1}}

\providecommand{\d}[2]{\diam(G/H)} 
\providecommand{\F}[1]{\Fit{#1}}

\providecommand{\ch}{\operatorname{char}}

\providecommand{\Pi}{P_{i}}

\usepackage{dutchcal}
 
\renewcommand{\c}[1]{ \operatorname{ cyc}  {(#1) }}

\renewcommand{\d}{ \delta}

\renewcommand{\ge}{\geqslant}
\renewcommand{\le}{\leqslant}
\renewcommand{\geq}{\geqslant}
\renewcommand{\leq}{\leqslant}
\newcommand{\eps}{\varepsilon}

 \renewcommand{\david}[1]{}
 \renewcommand{\c}[1]{}

  \newcommand{\soluble}[1]{}

\renewcommand\emph[1]{\textit{\textcolor{red}{#1}}}

  \newcommand{\paper}[1]{ {\color{blue}  } {} }

\renewcommand{\comment}[1]{ }

\makeatletter
\def\cref@thmoptarg[#1]#2#3#4{%
    \ifhmode\unskip\unskip\par\fi%
    \normalfont%
    \trivlist%
    \let\thmheadnl\relax%
    \let\thm@swap\@gobble%
    \thm@notefont{\fontseries\mddefault\upshape}%
    \thm@headpunct{.}%
    \thm@headsep 5\p@ plus\p@ minus\p@\relax%
    \thm@space@setup%
    #2%
    \@topsep \thm@preskip                
    \@topsepadd \thm@postskip            
    \def\@tempa{#3}\ifx\@empty\@tempa%
      \def\@tempa{\@oparg{\@begintheorem{\MakeUppercase #4}{}}[]}%
    \else%
      \refstepcounter[#1]{#3}%
      \@namedef{cref@#3@alias}{#1}%
      \def\@tempa{\@oparg{\@begintheorem{\MakeUppercase #4}{\csname the#3\endcsname}}[]}%
    \fi%
    \@tempa}%
\makeatother

\begin{document}

\begin{abstract}
Tointon and the author conjectured that, for a finitely generated residually finite group, virtual nilpotence is equivalent to the condition that the diameters of its finite coset spaces admit a uniform polynomial lower bound in terms of their sizes. We first verify this conjecture for the class of finitely generated soluble groups. We then prove that this polynomial lower bound condition implies that the group has a finite-index subgroup whose finite quotients are all soluble. An immediate consequence of these two results is the verification of the conjecture for finitely generated linear groups. In addition, we establish the same conclusion for certain finitely generated abelian-by-cyclic groups under the weaker assumption that their finite quotients satisfy this polynomial lower bound condition.
\end{abstract}

\maketitle  

\tableofcontents

\section{Introduction} \label{into}
The relationship between a group's algebraic structure and its geometry is a central theme in geometric group theory. A foundational result is Gromov's Theorem, which states that a finitely generated group has polynomial growth if and only if it is virtually nilpotent. Prior to this, Wolf showed that a polycyclic group of subexponential growth is virtually nilpotent, while Milnor proved that a finitely generated soluble group of subexponential growth must be polycyclic. Together, these results form the Milnor--Wolf Theorem: a finitely generated soluble group either has polynomial growth (and is thus virtually nilpotent) or has exponential growth. In recent years, several shorter proofs of Gromov's theorem have appeared (see, for instance, \citep{kleiner} and \citep{ozawa}). These proofs make essential use of Tits' alternative \citep{tits} for linear groups, which yields Gromov's theorem for linear groups as a consequence.

In \citep{Guo_Tointon}, Tointon and the author introduce two notions of uniform flatness for residually finite groups—two properties conjectured to be equivalent to virtual nilpotency. Inspired by the development of Gromov’s theorem, this paper addresses the conjecture in the soluble and linear case. We begin by recalling some definitions. A sequence $(N_k)_{k>0}$ of finite-index subgroups of a group $G$ is called a \emph{filtration} if it is strictly decreasing, that is, $N_{k+1} < N_k$ for all $k$. If, in addition, each $N_k$ is normal in $G$, we call $(N_k)_{k>0}$ a \emph{normal filtration}. Throughout this paper, we denote by $G$ an infinite group generated by a finite subset $S$, and we assume that $G$ admits a filtration. In particular, a group $G$ is residually finite if it admits a filtration with trivial intersection. It is straightforward to see that every infinite finitely generated soluble group admits a filtration.
We say that a group $G$ has \emph{uniformly $\alpha$-almost flat quotients} ($u.q.(\alpha)$) if there exists $\varepsilon > 0$ such that for every finite-index normal subgroup $N \trianglelefteq G$, we have:
\[
\operatorname{diam}_S(G/N) \ge \varepsilon [G:N]^\alpha.
\]
Similarly, $G$ has \emph{uniformly $\alpha$-almost flat coset spaces} ($u.c.(\alpha)$) if there exists $\varepsilon > 0$ such that for every finite-index subgroup $H \le G$, we have:
\begin{equation}\label{diam_lower_bound}
\operatorname{diam}_S(G/H) \ge \varepsilon [G:H]^\alpha.
\end{equation}

We say that $G$ has \emph{uniformly almost flat quotients} (u.q.) if there exists some $\alpha \in (0, 1]$ such that $G$ satisfies $u.q.(\alpha)$, and that $G$ has \emph{uniformly almost flat coset spaces} (u.c.) if there exists some $\alpha \in (0, 1]$ such that $G$ satisfies $u.c.(\alpha)$. Here is a partial statement of the conjecture on coset spaces from \citep{Guo_Tointon}.

\begin{conjecture}[{\citealp[Conjecture~1.7]{Guo_Tointon}}]\label{u_c_conj}
A finitely generated residually finite group has the u.c.\ property if and only if it is virtually nilpotent.
\end{conjecture}
The equivalence between the u.c.\ property and virtual nilpotency for polycyclic groups is proved in \citep{Guo_Tointon}.

\begin{thm}[{\citep[Theorem~1.5 and Lemma~2.2]{Guo_Tointon}}] \label{thm:poly} \label{polycyclic_by_finite_group_uda_iff_vn}
Suppose $G$ is an infinite polycyclic group. Then $G$ has the u.c.\ property if and only if $G$ is virtually nilpotent.
\end{thm}

The reader will notice that we generalise the notion of uniform almost flatness from \citep{Guo_Tointon} to groups that admit a filtration. This is motivated by the following facts:

\begin{enumerate}[(i)]
\item A finitely generated soluble group has polynomial growth if and only if all its residually finite quotients have polynomial growth; we will examine this further in \cref{subsec_core}.
\item Any finitely generated group with the u.q.\ property has a finite-index subgroup all of whose finite quotients are soluble, which we will prove in \cref{last_section}.
\end{enumerate}

Here is our first main theorem, which concerns the soluble case of \cref{u_c_conj}.

\begin{thm} \label{main_thm}
A finitely generated soluble group has the u.c.\ property if and only if it is virtually nilpotent.
\end{thm}

In the Milnor--Wolf Theorem, the finitely generated soluble case is reduced to the abelian-by-polycyclic case, and the abelian-by-polycyclic case is handled via a straightforward counting argument (see, for example, \citep[Theorem 14.31]{ggt}). In contrast, the proof of \cref{main_thm} is significantly more involved. 
We will refine the reduction process further by reducing the problem to the abelian-by-abelian case where the extension splits. We then divide our analysis into two classes of  soluble groups: those with finite subgroup rank and those with infinite subgroup rank. Here and throughout this paper, we say a group $G$ has \textit{finite subgroup rank} if there exists a positive integer $r$ such that every finitely generated subgroup of $G$ can be generated by at most $r$ elements. Conversely, $G$ is said to have \textit{infinite subgroup rank} if no such bound exists.

In the finite subgroup rank case, although the group structure is well understood, constructing a sequence of finite coset spaces that violates \eqref{diam_lower_bound} is highly delicate. This relies, for instance, on a deep result of Breuillard, Green, and Tao on the structure of approximate groups \citep{bgt}. In contrast, for groups of infinite subgroup rank, the structure is more complicated, but producing such a sequence is considerably more straightforward; in fact, we obtain a sharp bound on the diameter of finite coset spaces in terms of their index.

We next study general finitely generated groups satisfying the uniform almost flatness condition. Recent work of Eberhard, Maini, Sabatini, and Tracey studied finitely generated groups whose finite quotients do not have non-abelian composition factors of arbitrarily large Lie rank  \citep{eberhard2026diameterboundsarbitraryfinite}. Using a similar idea, we show that a finitely generated group with u.q. has a finite-index subgroup all of whose finite quotients are soluble. The author is grateful to Sean Eberhard for explaining this argument. As a consequence, we obtain the following two results.

\begin{thm} \label{main_thm_residually_soluble}
Let $G$ be a residually finite group with the u.q. property. Then $G$ is virtually residually soluble.
\end{thm}

\begin{thm} \label{main_thm_linear}
Let $G$ be a finitely generated linear group. Then $G$ has the u.c.\ property if and only if $G$ is virtually nilpotent.
\end{thm}

We further establish the following results concerning the u.q.\ property in two classical examples of non-polycyclic soluble groups. It is a well-known result that finitely generated abelian-by-cyclic groups are linear.

\begin{thm}
\label{main_thm_bs1_k} \label{main_bs}
For $k \geq 2$, the Baumslag--Solitar group $\mathrm{BS}(1,k) \cong  \mathbb{Z}\left[\frac{1}{k}\right] \rtimes_k \Z $ does not have the u.q. property.
\end{thm}

\begin{thm}
\label{main_thm__fg_abelian_wreath_not_uq}
Let $A$ be a finitely generated non-trivial abelian group. Then the wreath product $G = A \wr \mathbb{Z}$ does not have the u.q. property.
\end{thm}

The structure of the paper is as follows. In Section 2, we recall basic results in group theory, including a review of the uniformly almost flat property. In Sections 3 and 4, we treat the finitely generated abelian-by-abelian case with finite and infinite subgroup rank, respectively, under the assumption that the extension splits. In Section 5, we show that the general soluble case can be reduced to the abelian-by-abelian case with a split extension. This reduction relies on some significant work of Robinson and Wilson \citep{Robin_Wilson_jnp}, using techniques from group cohomology \citep{MR679154} and the theory of modules over polycyclic group rings \citep{peter_hall_finiteness_certain_soluble_groups}. In Section 6, we study general finitely generated groups satisfying the uniform almost flatness condition, where we prove \cref{main_thm_residually_soluble} and \cref{main_thm_linear}; this then naturally leads to a discussion of how to generalise \cref{main_thm} to general finitely generated residually finite groups. Finally, in Section 7, we prove \cref{main_thm_bs1_k} and \cref{main_thm__fg_abelian_wreath_not_uq}.

\subsection*{Acknowledgments} 
I would like to thank my PhD supervisor, Matthew Tointon, for introducing me to this project and encouraging me to complete it after my PhD. I am grateful to Peter Kropholler for his kindness and valuable help; his expertise greatly streamlined the proof of the main theorem, especially during a period when I had limited contact with the academic community. I also sincerely thank Sean Eberhard for his exceptional patience in sharing his insights on finite groups with me, for hosting me during a week-long research visit, and for his continued discussions and support throughout this project. Finally, I am grateful to my parents and to God for their support during the completion of this work.

\section{Some group theory}
\subsection{The core of a group} \label{subsec_core}

We begin by addressing a point mentioned in the introduction section: a finitely generated soluble group is not necessarily residually finite, yet whether it is  virtually nilpotent can be deduced from its residually finite quotients.

\begin{prop} \label{prop:rf_nilpotent}
A finitely generated soluble group $G$ is virtually nilpotent if and only if every residually finite quotient of $G$ is virtually nilpotent.
\end{prop}

\begin{rem}
Note that this property does not hold for all finitely generated groups. For instance, there exist infinite finitely generated simple groups that exhibit exponential growth. A more elementary example is the wreath product $A_5 \wr \mathbb{Z}$. A straightforward exercise shows that this group has exponential growth, while all its residually finite quotients are cyclic.
\end{rem}

We will prove a slightly stronger statement than that of \cref{prop:rf_nilpotent}. Recall that the (residual finiteness) \emph{core} $\operatorname{Core}(G)$ of a countable group $G$ is defined as the intersection of all finite-index subgroups of $G$, and is often referred to as the finite residual of $G$. A group $G$ is residually finite if and only if $\operatorname{Core}(G) = \{1\}$.

\begin{lemma} \label{v_poly_case}
Let $G$ be a finitely generated group. Suppose $G / \operatorname{Core}(G)$ is virtually polycyclic. Then $\operatorname{Core}(G)$ is perfect.
\end{lemma}

\begin{proof}
By the Roseblade--Jategaonkar result \citep{Roseblade, JAT}, finitely generated abelian-by-polycyclic groups are residually finite. Since a group is residually finite if and only if it is virtually residually finite, it follows that abelian-by-(virtually polycyclic) groups are also residually finite. Consider the following quotient:
\[
\frac{G / [\operatorname{Core}(G), \operatorname{Core}(G)]}{\operatorname{Core}(G) / [\operatorname{Core}(G), \operatorname{Core}(G)]} \cong \frac{G}{\operatorname{Core}(G)}.
\]
The subgroup $\operatorname{Core}(G) / [\operatorname{Core}(G), \operatorname{Core}(G)]$ is abelian, and by hypothesis, the quotient $G / \operatorname{Core}(G)$ is virtually polycyclic. Thus, $G / [\operatorname{Core}(G), \operatorname{Core}(G)]$ is an abelian-by-(virtually polycyclic) group, which implies it is residually finite. Therefore we must have $[\operatorname{Core}(G), \operatorname{Core}(G)] = \operatorname{Core}(G).$
\end{proof}

The proof of \cref{prop:rf_nilpotent} follows immediately from the following corollary:

\begin{cor} \label{cor:core_nilpotent}
Let $G$ be a finitely generated soluble group. If $G / \operatorname{Core}(G)$ is virtually nilpotent, then $G$ is virtually nilpotent.
\end{cor}

\begin{proof}
Since every virtually nilpotent group is virtually polycyclic, \cref{v_poly_case} implies that $\operatorname{Core}(G)$ is a perfect group. 
Furthermore, since $\operatorname{Core}(G)$ is a subgroup of the soluble group $G$, it must itself be soluble. However, the only soluble perfect group is the trivial group; hence, $\operatorname{Core}(G) = \{1\}$. Hence, $G \cong G / \operatorname{Core}(G)$ is virtually nilpotent.
\end{proof}







\subsection{Some Basic properties of uniform almost flatness}\label{sec:basic}
\begin{lem}[{\citealp[Lemma~2.1]{Guo_Tointon}}] \label{lem:indep.gen.set}
The definitions of having uniformly $\alpha$-almost flat quotients and having uniformly $\alpha$-almost flat coset spaces do not depend on the choice of the generating set.
\end{lem}

\begin{lem}[{\citealp[Lemma~2.2]{Guo_Tointon}}] \label{lem:poly.growth} \label{poly_growth_implies_uniform_D_a}
Suppose $G$ is a group of polynomial growth of degree $d$. Then $G$ has uniformly $(1/d)$-almost flat coset spaces.
\end{lem}

Note that the property of having uniformly $\alpha$-almost flat coset spaces (u.c.\((\alpha)\)) clearly passes to quotients. Furthermore, \citep[Lemma~2.5]{Guo_Tointon} shows that this property also passes to finite-index subgroups. Inspired by this, we introduce the following definition for ease of notation: for a group $G$, let $\mathcal{C}(G)$ denote the class of all quotients of finite-index subgroups of $G$. Thus, \citep[Lemma~2.5]{Guo_Tointon} implies the following result:
\begin{lem} \label{uniform_D_alpha_passes_finte_index_subgroup}
Suppose $G $ has u.c.\((\alpha)\). Then every $H \in \mathcal{C}(G)$ also has u.c.\((\alpha)\).
\end{lem}

\subsection{Semidirect product}
Let $K$ and $H$ be groups, and let $K \rtimes H$ denote the semidirect product of $K$ by $H$ with respect to an action of $H$ on $K$. We write $h \cdot k$ for the image of $k \in K$ under the action of $h \in H$. In the case where the action is given by conjugation, this coincides with $h \cdot k = h k h^{-1} = k^h$. The multiplication in $K \rtimes H$ is given by
\[
(k_1, h_1) \cdot (k_2, h_2) = \big(k_1 \, (h_1 \cdot k_2), \, h_1 h_2\big).
\]
By default, we identify elements $k \in K$ and $h \in H$ with $(k,1)$ and $(1,h)$, respectively.

\subsection{ Miscellaneous Results}

We record the following standard closure property of the class \( \mathcal{C}(G) \).

\begin{thm} \label{finite_index_quotient}
Let $G$ and $H$ be groups. Suppose there exists a finite sequence $G_0, G_1, \dots, G_n$ such that $G_0 = G$, $G_n \cong H$, and for each $0 \le i < n$, $G_{i+1}$ is either a quotient or a finite-index subgroup of $G_i$. Then $H$ is a quotient of some finite-index subgroup $K \le G$.
\end{thm}
\begin{proof}
We argue by induction on $n$. The base case $n=1$ is clear. For the inductive step, assume the statement holds for $n$. Then there exists a finite-index subgroup $K \le G$ and a surjection $\pi: K \twoheadrightarrow G_n$. If $G_{n+1}$ is a quotient of $G_n$, say $\psi: G_n \twoheadrightarrow G_{n+1}$, then $\psi \circ \pi: K \twoheadrightarrow G_{n+1}$, i.e. $G_{n+1}$ is a quotient of $K$. If $G_{n+1} \le G_n$ with $[G_n : G_{n+1}] < \infty$, set $K' = \pi^{-1}(G_{n+1})$. Then $[G : K'] < \infty$ since the preimage of a finite-index subgroup under a group homomorphism has finite index. Moreover, $\pi(K') = G_{n+1}$, so $G_{n+1}$ is a quotient of $K'$. Thus, in all cases, $G_{n+1}$ is a quotient of a finite-index subgroup of $G$, completing the induction.
\end{proof}

\begin{lem}
    Let $G$ be an abelian-by-nilpotent group. If $H \le G$ is a subgroup, then $H$ is also abelian-by-nilpotent.
\end{lem}

\begin{proof}
By definition, there exists $A \trianglelefteq G$ such that $A$ is abelian and $G/A$ is nilpotent. Let $H \le G$ and set $K = H \cap A$. Then $K$ is abelian and normal in $H$. By the Second Isomorphism Theorem,
\[
H/(H \cap A) \cong HA/A \le G/A.
\]
Since $G/A$ is nilpotent and nilpotency is inherited by subgroups, it follows that $H/K$ is nilpotent. Hence $H$ is abelian-by-nilpotent.
\end{proof}
We immediately have the following:
\begin{cor} \label{class_closure_abelian_by_nilpotent}
Let $G$ be an abelian-by-nilpotent group. Then every group in $\mathcal{C}(G)$ is also abelian-by-nilpotent.
\end{cor}

\section{Finite subgroup rank }

We begin by investigating the u.c.\ property in finitely generated soluble groups with finite subgroup rank. An important class of such groups is the class of polycyclic groups, where the u.c.\ property was studied in preceding work \citep{Guo_Tointon}. On the one hand, there are similarities in the approach, as we develop tools from linear algebra and number theory before applying the strengthened version of Gromov’s Theorem by Breuillard, Green, and Tao \citep{bgt}. Indeed, while the study of polycyclic groups involves understanding general linear matrices with integer entries, the study of finitely generated soluble groups with finite subgroup rank will involve rational entries instead. On the other hand, we encounter a distinct situation due to the existence of infinitely generated subgroups. For example, the polycyclic case reduces to the abelian-by-cyclic case (see \citep[Lemma 8.21 ]{Guo_Tointon}), this reduction is not generally possible for finitely generated soluble groups with finite subgroup rank (see the proof of \cref{char_K_Z_m}).

We will  study finitely generated groups of the form $G = A \rtimes \mathbb{Z}^m$, where $A$ is an abelian group of finite rank. Since $\operatorname{Tor}(A)$ is a finite characteristic subgroup of $G$, we may assume without loss of generality that $A$ is torsion-free. For convenience, we define the class of groups $\mathcal{A}_{\mathrm{f.r.}}(n)$ as follows: $A \in \mathcal{A}_{\mathrm{f.r.}}(n)$ if $A$ is a torsion-free abelian group of finite rank $n$ such that it can be extended to a finitely generated group $G = A \rtimes \Gamma$, where $\Gamma$ is polycyclic.

\subsection{Classifying the groups in $\mathcal{A}_{\text{f.r.}}(n)$}

Given a torsion-free abelian group $A$, there is a standard way to view $A$ as a subgroup of a $\mathbb{Q}$-vector space. Define
\[
V := A \otimes_{\mathbb{Z}} \mathbb{Q},
\]
the tensor product of the abelian group $A$ (viewed as a $\mathbb{Z}$-module) with the field of rational numbers $\mathbb{Q}$. This construction is known as the \emph{rationalisation} of $A$. Since $A$ is torsion-free, the natural homomorphism
\[
\iota \colon A \longrightarrow V, \qquad a \longmapsto a \otimes 1,
\]
is injective. Thus, we may identify $A$ with a subgroup of the $\mathbb{Q}$-vector space $V$. The following standard result records several equivalent notions of rank for a torsion-free abelian group.

\begin{prop}[Equivalent definitions of rank]
Let $A$ be a torsion-free abelian group. The following are equivalent and define the
\emph{rank} of $A$:
\begin{enumerate}
  \item The cardinality of a maximal $\mathbb{Z}$-linearly independent subset of $A$ (equivalently,  the    subgroup rank of $A$.).
  \item The largest   $n \in \N \cup \{\infty\}$ such that $A$ contains a subgroup
        isomorphic to $\mathbb{Z}^n$.
  \item The dimension of the $\mathbb{Q}$-vector space $A \otimes_{\mathbb{Z}} \mathbb{Q}$.
  \item The least integer $r$ such that every finitely generated subgroup of $A$
        can be generated by at most $r$ elements.
\end{enumerate}
\end{prop}

By embedding $A$ into a vector space, we can immediately see that, after fixing a maximal $\mathbb{Z}$-linearly independent set of $A$, every automorphism of $A$ can be represented by a matrix $Q \in \mathrm{GL}(n,\mathbb{Q})$.

\begin{lem}[Automorphisms of finite-rank torsion-free abelian groups]  \label{auto_of_torsion_free_finite_rank}
Let $A$ be a torsion-free abelian group of rank $n$, and let $\{e_1, \dots, e_n\}$ be a maximal $\mathbb{Z}$-linearly independent subset of $A$. Then, via the natural embedding 
\[
A \hookrightarrow A \otimes_{\mathbb{Z}} \mathbb{Q},
\]
the set $\{e_1, \dots, e_n\}$ forms a $\mathbb{Q}$-basis for the vector space $V := A \otimes_{\mathbb{Z}} \mathbb{Q}$. Consequently, every automorphism of $A$ extends uniquely to a $\mathbb{Q}$-linear automorphism of $V$ and can thus be represented by a matrix $M \in \mathrm{GL}_n(\mathbb{Q})$ with respect to this basis.
\end{lem}

\begin{proof}
Let $\phi \in \operatorname{Aut}(A)$. Define a map 
\[
\Phi \colon V \to V, \quad \Phi(a \otimes q) := \phi(a) \otimes q \quad \text{for all } a \in A, \, q \in \mathbb{Q}.
\]
Since $\phi$ is a group homomorphism, $\Phi$ is a well-defined $\mathbb{Q}$-linear transformation. Moreover, as $\phi$ is an automorphism, it possesses an inverse $\phi^{-1}$; the map induced by $\phi^{-1}$ serves as the inverse to $\Phi$. Thus, $\Phi \in \operatorname{GL}(V)$.

By fixing the basis $\{e_1, \dots, e_n\}$ of $V$, the linear transformation $\Phi$ is uniquely represented by a matrix $M \in \operatorname{GL}_n(\mathbb{Q})$, where the $j$-th column of $M$ consists of the coordinates of $\Phi(e_j)$ with respect to this basis. Since $\phi$ is the restriction of $\Phi$ to $A$, it follows that $\phi$ is represented by the rational matrix $M$.
\end{proof}

\begin{defn}
A group $G$ is said to have the \emph{bounded generation property} if there exists a finite subset $\{t_1, \ldots, t_m\} \subset G$ such that every $g \in G$ can be written as
\[
g = t_1^{k_1} t_2^{k_2} \cdots t_m^{k_m},
\]
where $k_1, k_2, \ldots, k_m \in \mathbb{Z}$.
\end{defn}

A typical example of a group with the bounded generation property is a polycyclic group.

\begin{thm}[Matrix representation of groups in $\mathcal{A}_{\mathrm{f.r.}}(n)$] \label{thm:matrix_rep}
Let $G = A \rtimes \Gamma$ be a finitely generated group where $A$ is a torsion-free abelian group of rank $n$ and $\Gamma$ is a polycyclic group boundedly generated by $\{x_1, \dots, x_m\}$. Then there exist matrices $M_1, \dots, M_m \in \mathrm{GL}_n(\mathbb{Q})$ such that
\[
A \cong \mathcal{M} := \left\langle M_1^{k_1} \cdots M_m^{k_m} v \;\middle|\; k_i \in \mathbb{Z}, \, v \in \mathbb{Z}^n \right\rangle \subseteq \mathbb{Q}^n,
\]
and the action of $\Gamma$ on $\mathcal{M}$ is given by $x_i \cdot w = M_i w$. In the particular case where $\Gamma \cong \mathbb{Z}^m$, the matrices $M_1, \dots, M_m$ pairwise commute.
\end{thm}

\begin{proof}
Since $\Gamma$ is finitely presented, \citep[Lemma~7.30]{ggt} implies that there exists a finite set $Y \subseteq A$ such that every element of $A$ can be expressed as a product of conjugates of elements of $Y$ by elements of $\Gamma$. As $\Gamma$ is boundedly generated, these conjugates may be written in the form $\mathbf{x}^{\mathbf{p}} \cdot y$, where $y \in Y$, $\mathbf{p} \in \mathbb{Z}^m$, and $\mathbf{x}^{\mathbf{p}} = x_1^{p_1} \cdots x_m^{p_m}$.

Let $H_k$ be the subgroup of $A$ generated by all elements of the form $\mathbf{x}^{\mathbf{p}} \cdot y$ with $|p_i| \le k$. Then $(H_k)$ is an increasing sequence of subgroups whose ranks are bounded above by $n$, so the ranks stabilise and eventually equal $n$. Let $H = H_\kappa$ be such a subgroup. Then $H$ is a finitely generated torsion-free abelian group of rank $n$, so we may choose a basis $\{r_1, \dots, r_n\}$ with $H = \langle r_1, \dots, r_n \rangle \cong \mathbb{Z}^n$.

Each automorphism $a \mapsto x_i \cdot a$ of $A$ extends uniquely to a $\mathbb{Q}$-linear automorphism of $V := A \otimes_\mathbb{Z} \mathbb{Q}$, and hence is represented by a matrix $M_i \in \mathrm{GL}_n(\mathbb{Q})$ with respect to the chosen basis by \cref{auto_of_torsion_free_finite_rank}.

Define
\[
\mathcal{M} = \left\langle M_1^{k_1} \cdots M_m^{k_m} e_j \;\middle|\; 1 \le j \le n,\; k_i \in \mathbb{Z} \right\rangle \subseteq \mathbb{Q}^n,
\]
where $\{e_1, \dots, e_n\}$ is the standard basis. Then $\Gamma$ acts on $\mathcal{M}$ via $x_i \cdot w = M_i w$.

Define $\varphi : A \to \mathcal{M}$ by sending $r_j \mapsto e_j$ and extending equivariantly with respect to the action of $\Gamma$. Explicitly, for $a = \mathbf{x}^{\mathbf{k}} \cdot r_j$, set
\[
\varphi(a) = M_1^{k_1} \cdots M_m^{k_m} e_j.
\]
Then $\varphi$ is a surjective homomorphism.

To see that $\varphi$ is injective, note that since $\operatorname{rank}(H) = \operatorname{rank}(A)$, for every $a \in A$ there exists $n_a \in \mathbb{N}$ such that $n_a a \in H$. If $\varphi(b)=0$, then $\varphi(n_b b)=0$, and since $\varphi|_H$ is injective, we deduce $n_b b=0$. As $A$ is torsion-free, this implies $b=0$. Hence $\varphi$ is an isomorphism, as required.
\end{proof}

We will be using \cref{thm:matrix_rep} in the case \( \Gamma \cong \mathbb{Z}^m \). Let \( M_1, \dots, M_m \in \mathrm{GL}_n(\mathbb{Q}) \) be commuting matrices, and define
\begin{equation*}
K(M_1, \dots, M_m)
:= \left\langle M_1^{k_1} \cdots M_m^{k_m} v \;\middle|\; k_i \in \mathbb{Z},\ v \in \mathbb{Z}^n \right\rangle
\subseteq \mathbb{Q}^n.
\end{equation*}

We form the semidirect product
\[
G = K(M_1, \dots, M_m) \rtimes \mathbb{Z}^m
= K(M_1, \dots, M_m) \rtimes \langle t_1, \dots, t_m \rangle,
\]
where \( \mathbb{Z}^m \) acts via \( t_i v t_i^{-1} = M_i v \) for all \( v \in K(M_1, \dots, M_m) \) and \( i = 1, \dots, m \).

Let \( D \) be a set of primes. We define the ring of \( D \)-integers by
\[
\mathbb{Z}_D
=
\left\{
\frac{a}{b} \in \mathbb{Q}
\;\middle|\;
a \in \mathbb{Z},\ b \in \mathbb{Z}_{>0},\ \text{all prime factors of } b \text{ lie in } D
\right\}.
\]

It is clear that there exists a finite set of primes \( D \), such that each entry of every \( M_i \) has denominator with factors only in \( D \), i.e. lies in \( \mathbb{Z}_D \). When we want to emphasise the possible prime divisors of the denominators of the matrices, we may write \( M_i \in \mathrm{GL}(n,\mathbb{Z}_D) \). It is clear that \( K(M_1, \dots, M_m) \subseteq \mathbb{Z}_D^n \).

The following result tells us when $ K(M_1, \dots, M_m) \rtimes \mathbb{Z}^m$ is actually a polycyclic group.

\begin{prop}[Criteria for Finite Generation of $\mathcal{M}$] \label{Finite_generation_criteria}
Let $M_1, \dots, M_m \in \mathrm{Mat}_{n \times n}(\mathbb{Q})$ be pairwise commuting matrices, and suppose that the characteristic polynomial of each $M_i$ has integer coefficients. Then:

\begin{enumerate}[(i)]
    \item For each $i$, the sequence $\{M_i^k\}_{k \ge 0}$ has bounded denominators.  
    \item If, in addition, the characteristic polynomial of each $M_i^{-1}$ also has integer coefficients, then the group
    \[
    \mathcal{M} = \Big\langle \left( \prod_{i=1}^m M_i^{k_i} \right) v \;\Big|\; k_i \in \mathbb{Z},\ v \in \mathbb{Z}^n \Big\rangle \subseteq \mathbb{Q}^n
    \]
    is finitely generated as an abelian group.
\end{enumerate}
\end{prop}

\begin{proof}
We prove (i); (ii) follows by a similar argument. Suppose that every coefficient of $\chi(M_i)$ is an integer. By the Cayley--Hamilton theorem, $M_i^n$ is an integer linear combination of $I, M_i, \dots, M_i^{n-1}$. By induction, every power $M_i^k$ for $k \ge 0$ lies in the $\mathbb{Z}$-span of $\{I, M_i, \dots, M_i^{n-1}\}$. This span is a finitely generated $\mathbb{Z}$-module, hence the sequence $\{M_i^k\}_{k \ge 0}$ has bounded denominators.
\end{proof}

\subsection{Some  number theory}
We will first define modular arithmetic for all rational numbers. Let $a, b \in \Z$, and let $p$ be a prime that does not divide $b$. We say $\frac{a}{b} \equiv_p k$ if $k b \equiv_p a$. Given a monic polynomial $P(x) \in \qx$, we define $\lcm(P)$ as the lowest common multiple of the denominators of all coefficients of $P$. Given a prime $p \nmid \lcm(P)$, we say that $P(x)$ splits over $\Fp$ if it factors as a product of linear factors over $\Fp$.  In the special case where we work with a rational matrix $Q$ and take $P(x)=\chi_Q(x)$, we may be slightly sloppy and instead require that $p \nmid \operatorname{lcm}(Q)$, without mentioning this explicitly; this condition in particular implies that $p \nmid \operatorname{lcm}(P)$.

For example, consider the rational polynomial $x^2 + \frac{2}{3}$. Note that $3^{-1} \equiv_5 2$; hence we have $x^2 + \frac{2}{3}$ splits over $\mathbb{F}_5$ as $(x+4)(x+1)$. Also,  
\[
(x+4)(x+1) = (x^2 + \tfrac{2}{3}) + \Big(5x + \tfrac{10}{3}\Big),
\]  
and the remainder $5x + \tfrac{10}{3}$ indeed vanishes over $\mathbb{F}_5$.

\begin{lem}\label{split_over_inf_primes_with_no_zero_roots}
Given any monic $P(x)$ in $\qx$ with $P(0)\ne 0$, the set of primes $p$ for which $P(x)$ splits completely over $\mathbb{F}_p$ and has no zero root has positive density. In particular, there are infinitely many such primes.
\end{lem}

\begin{proof}
It is well-known that every non-constant polynomial in $\Z[X]$ splits completely over infinitely many primes, with positive density given by the Frobenius Density Theorem (see~\citep{dasfrob} or~\citep[Page~11]{Frobenius_eng}).
 In particular, there exist infinitely many primes $p$ coprime to $ |\lcm(P)^2 P(0)|$ such that $\lcm(P) P(x) \in \Z[X]$ splits over $\F_p$. For every such prime, it is clear that $P(x)$ splits over $\F_p$ with no zero roots.
\end{proof}

We next show that  \cref{split_over_inf_primes_with_no_zero_roots} extends to the simultaneous splitting of any finite collection of monic rational polynomials.

\begin{lem} \label{product_of_poly}
Let $\F$ be a field.
Let $f(x), g(x)$ be two non-zero polynomials in $ \F[x]$ such that  $f(x)g(x)$ splits over $\F$. Then both $f(x)$ and $g(x)$ split over $\F$. \
\end{lem}

\begin{proof}
Since $\F$ is a field, $\F[x]$ is a unique factorisation domain, from which the result follows immediately.
\end{proof}

\begin{lem} [Infinite splitting primes for finite collection of polynomials] \label{mult_char_polys}
Suppose $f_1(x), \ldots, f_m(x)$ are non-zero monic polynomials in $\Q[X]$ with $f_i(0) \neq 0$, then there are infinitely many primes $p$ such that every $f_i(x)$ split over $\F_p$ with no zero roots.
\end{lem}

\begin{proof}
This follows from \cref{split_over_inf_primes_with_no_zero_roots} and \cref{product_of_poly}.
\end{proof}

We now turn our attention to the behaviour of the multiplicities of the roots of a single polynomial $P(x)$ when reduced modulo $p$ for $p$ in the set of splitting primes. Let $P(x)$ be a monic polynomial in  $ \qx$, we will let $\Pr(P)$ be the set of primes $p$ such that $P(x)$ splits over $\F_p$ with no zero roots. Let $p \in \Pr(P)$, we define $\lambda(P,p)$ as the lowest common multiple of the multiplicative orders of the set of  roots of $P(x)$ in $\F_p$.

\begin{thm} \label{diverges_mul_order_rat} [The multiplicative order of the roots]
Let $P(x)$ be a polynomial in $\mathbb{Q}[x]$ with a root that is not a root of unity. Then any infinite subset of $\{\lambda(P,p) \mid p \in \Pr(P)\}$ is unbounded; equivalently, the sequence $\{\lambda(P,p) \mid p \in \Pr(P)\}$ diverges to infinity.
\end{thm}

In \citep[Corollary 7.12]{Guo_Tointon}, an analogous result was established for monic integer matrices possessing at least one root with an absolute value not equal to $1$. Here, we extend this to rational polynomials, though we must explicitly require the existence of a root that is not a root of unity. This distinction is necessary for the following reason: if $f(x) \in \mathbb{Z}[x]$ is monic, Kronecker's Theorem implies that all its roots satisfy $|\alpha|=1$ if and only if those roots are roots of unity. However, this equivalence fails for monic polynomials in $\mathbb{Q}[x]$. For example, the monic polynomial $x^2 - \tfrac{4}{3}x + 1$ has roots on the unit circle ($|\alpha|=1$) that are not roots of unity.
\begin{proof}
We prove the contrapositive. Suppose there exists an infinite subset \( P' \subseteq \Pr(P) \) such that \( \lambda(P,p) \le M \) for all \( p \in P' \). We will show that every root of \( P \) is an \( M! \)-th root of unity. For each \( p \in P' \), write \( P(x) = \prod_{i=1}^n (x - a_{i,p}) \).  Then we have \( a_{i,p}^{M!} \equiv_p 1 \), so \( x - a_{i,p} \mid x^{M!} - 1 \pmod{p} \), hence
\[
P(x) \mid (x^{M!} - 1)^n \pmod{p}.
\]

Applying the Euclidean division algorithm in \( \mathbb{Q}[x] \), we write
\[
(x^{M!} - 1)^n = P(x) Q(x) + R(x),
\]
where \( Q(x), R(x) \in \mathbb{Q}[x] \) and \( \deg(R) < \deg(P) = n \). There are infinitely many \( p \in P' \) with
\[
p > \max\{\operatorname{lcm}(Q), \operatorname{lcm}(R)\}.
\]
For such \( p \), we have
\[
R(x) \equiv 0 \pmod{p}.
\]
Therefore, we must have \( R(x) = 0 \) over \( \mathbb{Q}[x] \), so \( P(x) \mid (x^{M!} - 1)^n \) in \( \mathbb{Q}[x] \). Every root of \( (x^{M!} - 1)^n \) is a root of unity, so every root of \( P \) is a root of unity, completing the proof.
\end{proof}
We will later apply these number-theoretic results to the study of sequences of matrices in \( \mathrm{GL}_n(\mathbb{Q}) \), where we use the fact that there are infinitely many primes for which all the characteristic polynomials split. Let \( D \) be a finite set of primes and let \( p \notin D \) be a prime. Given a matrix \( M \in \mathbb{Z}_D^{n \times m} \), we denote by \( \tilde{M}_p \in \mathbb{F}_p^{n \times m} \) the reduction of \( M \bmod p \). When it is clear from the context which prime we are referring to, we often drop the subscript \( p \) and simply write \( \tilde{M} \).

\subsection{Finite index subgroups}


Given an additive subgroup $A$ of $\mathbb{Q}^n$ and $m \in \N$, it follows from \citep[Exercise~92.5]{infinite_abelian_groups} that $mA$ has finite index in $A$. We give a short proof for the case we are interested in and provide the explicit index.


\begin{lem}[Index of subgroups in $\mathbb{Z}_D^n$] \label{index}
Let $D$ be a set of primes with $p \notin D$. Let $A \le \mathbb{Z}_D^n$ be a subgroup of rank $n$ containing $\mathbb{Z}^n$. Let $\{e_i\}$ be the standard basis of $\mathbb{Z}^n$. Consider the map $$\Phi: A \to (\mathbb{Z}/p\mathbb{Z})^n$$ defined by extending $e_i \mapsto \bar{e}_i$. Then $\Phi$  is a surjective homomorphism with $\ker(\Phi) = pA$.
\end{lem}

\begin{proof}
The map $\Phi$ is surjective since $A \supseteq \mathbb{Z}^n$ and $\Phi(\mathbb{Z}^n) = (\mathbb{Z}/p\mathbb{Z})^n$. We show $\ker(\Phi) = pA$. The inclusion $pA \subseteq \ker(\Phi)$ is clear. For the reverse, let $\underline{v} \in \ker(\Phi)$. We can write $\underline{v} = \sum_{i=1}^n \frac{m_i}{y_i} \underline{e}_i$ where $p \nmid y_i$ (since $y_i$ has primes factors only in $D$). Since $\underline{v} \in \ker(\Phi)$, it follows that $p \mid m_i$ for all $i$, so $m_i = p k_i$ for some $k_i$. Then:
\[ \underline{v} = p \left( \sum_{i=1}^n \frac{k_i}{y_i} \underline{e}_i \right) \]
Let $\underline{u} = \sum \frac{k_i}{y_i} \underline{e}_i$,  we will show $\underline{u} \in A$. Let $s = \operatorname{lcm}(y_i)$, then $s\underline{u} = \underline{b}$ for some $\underline{b} \in \mathbb{Z}^n \subseteq A$. Since $p \nmid s$, Bézout's identity implies there exist $x, w \in \mathbb{Z}$ such that $xs + wp = 1$. Thus:
\[ \underline{u} = (xs + wp)\underline{u} = x(s\underline{u}) + w(p\underline{u}) = x\underline{b} + w\underline{v} \]
Since $\underline{b} \in A$ and $\underline{v} \in A$, and $A$ is a subgroup, it follows that $\underline{u} \in A$. Therefore, $\underline{v} = p\underline{u} \in pA$, which proves $\ker(\Phi) \subseteq pA$.
\end{proof}

\subsection{Some linear algebra}

 Suppose that $T$ is a linear map over $V = \F^n$, where $\F$ is a field. We define the \emph{eigenspace} of $T$ for $\l$ as
\[
E_{\l}(T) = \{ x \in \F^n \mid (T - \l I) x = 0 \}.
\]
If the characteristic polynomial of the linear map $T$ splits over $\F$, then we know $T$ has an eigenvector. In particular, we have the following observation.

\begin{lem} \label{eigenvector_in_invariant_subspace}
Suppose $\ch(M)$ splits over $\F$. Let $W$ be a non-trivial $M$-invariant subspace. Then $W$ contains an eigenvector of $M$.
\end{lem}

\begin{lem} [Commuting matrices and invariant subspace]\label{commute_basis}
Let $T_1, \dots, T_m$ be a set of automorphisms of the vector space $\F^n$ that pairwise commute. Suppose that the characteristic polynomials $\ch(T_1), \dots, \ch(T_m)$ split over $\F$. Let $\mu_1$ be an eigenvalue of $T_1$. Then:
\begin{enumerate}[(i)]
    \item There exist eigenvalues $\mu_2, \dots, \mu_m$ of $T_2, \dots, T_m$, respectively, such that 
    \[ \bigcap_{i=1}^{m} E_{\mu_i}(T_i) \neq \{0\}, \]
i.e., the $T_i$ share a simultaneous eigenvector in the eigenspace $E_{\mu_1}(T_1)$ corresponding to $\mu_1$.
   \item Let $(\lambda_1, \dots, \lambda_n)$ be an ordered list of the eigenvalues of $T_1$,  then there exists a basis $\mathcal{B}$ of $\F^n$ such that:
    \begin{enumerate}
        \item For each $i \in \{1, \dots, m\}$, the matrix $[T_i]_\mathcal{B}$ is upper triangular.
        \item The diagonal entries of $[T_1]_\mathcal{B}$ are exactly $(\lambda_1, \dots, \lambda_n)$.
    \end{enumerate}
\end{enumerate}
\end{lem}

\begin{proof}
\begin{enumerate}[(i)]
\item We will prove the statement via induction on $m$. The case for $m=1$ is given in the assumption. Suppose the statement holds for $m-1$. i.e. there exists $\mu_2, \ldots,\mu_{m-1}$, such that $ \bigcap_{i=1} ^{m-1} E_{\mu_i}(T_i)  \neq \{0\}$. Recall the fact that commuting linear maps preserve each other's eigenspaces, it follows that $$T_m \left( \bigcap_{i=1} ^{m-1} E_{\mu_i}(T_i) \right) = \bigcap_{i=1} ^{m-1} T_m \left(E_{\mu_i}(T_i) \right) \subseteq \bigcap_{i=1} ^{m-1} E_{\mu_i}(T_i).$$ According to \cref{eigenvector_in_invariant_subspace}, $T_m$ restricted to this intersection has an eigenvalue $\mu_m$ with an eigenvector in $\bigcap_{i=1} ^{m-1} E_{\mu_i}(T_i)$, which is to say that $ \bigcap_{i=1} ^{m} E_{\mu_i}(T_i) \neq \{0\}$. 

\item We proceed by induction on $n$. The case $n=1$ is trivial. Assume the statement holds for dimension $n-1$. Let $(\lambda_1, \dots, \lambda_n)$ be the ordered eigenvalues of $T_1$. By part (i), there exists a common eigenvector $\underline{v}_1 \neq 0$ such that $T_i \underline{v}_1 = \mu_i \underline{v}_1$ for all $i$, with $\mu_1 = \lambda_1$.

Let $U = \operatorname{span}\{\underline{v}_1\}$. Consider the quotient space $\F^n/U$ of dimension $n-1$. Each $T_i$ induces a map $\bar{T}_i$ on $\F^n/U$. Since $\chi(T_i)$ splits, its factor $\chi(\bar{T}_i)$ also splits, with $\chi(T_i) = (x-\mu_i)\chi(\bar{T}_i)$. By the induction hypothesis, there is a basis $\mathcal{B}' = \{\underline{v}_2 + U, \dots, \underline{v}_n + U\}$ for $\F^n/U$ that upper triangularizes all $\bar{T}_i$, and the diagonal entries of $[\bar{T}_1]_{\mathcal{B}'}$ are exactly $(\lambda_2, \dots, \lambda_n)$.

Let $\mathcal{B} = \{\underline{v}_1, \underline{v}_2, \dots, \underline{v}_n\}$. Since $T_i \underline{v}_1 = \mu_i \underline{v}_1$, the first column of $[T_i]_\mathcal{B}$ has $\mu_i$ on the diagonal and zeros elsewhere. For $j > 1$, the fact that $\bar{T}_i(\underline{v}_j + U) \in \operatorname{span}\{\underline{v}_2 + U, \dots, \underline{v}_j + U\}$ implies $T_i \underline{v}_j \in \operatorname{span}\{\underline{v}_1, \dots, \underline{v}_j\}$. Thus, $[T_i]_\mathcal{B}$ is upper triangular, and the diagonal entries of $[T_1]_\mathcal{B}$ are $(\lambda_1, \dots, \lambda_n)$.
\end{enumerate}
\end{proof}

\begin{cor}\label{index_p_subgroup_rat}
Let \( D \) be a finite set of primes, and let \( M_1, \dots, M_m \in \mathrm{GL}(n, \mathbb{Z}_D) \) be a set of pairwise commuting matrices. Let \( K = K(M_1, \dots, M_m) \). Let \( p \) be a prime not in \( D \) such that each characteristic polynomial \( \chi(M_i) \) splits over \( \mathbb{F}_p \). Let \( \lambda \) be an eigenvalue of \( M_1 \) in \( \mathbb{F}_p \). Then there exists a basis \( \tilde{\mathcal{B}} = \{\underline{v}_1, \dots, \underline{v}_n\} \) of the quotient space \( K / pK \cong \mathbb{F}_p^n \) such that each \( (M_i)_p \) preserves the subspace \( \langle \underline{v}_1, \dots, \underline{v}_{n-1} \rangle \), and
\[
(M_1)_p \;\underline{v}_n = \lambda \, \underline{v}_n + \underline{w}
\]
for some \( \underline{w} \in \langle \underline{v}_1, \dots, \underline{v}_{n-1} \rangle \).

In particular, let \( H_p \) be the subgroup of \( K \) containing \( pK \) that corresponds to the subspace \( \langle \underline{v}_1, \dots, \underline{v}_{n-1} \rangle \) under the quotient map. Then:
\begin{enumerate}
    \item \( H_p \) has index \( p \) in \( K \);
    \item each \( M_i \) preserves \( H_p \);
    \item letting \( \bar{M}_1 \) be the induced endomorphism of \( M_1 \) on \( K/H_p \cong \mathbb{Z}_p \), for every \( w \in K/H_p \), viewed as an element of \( \mathbb{Z}_p \), we have    \[
    \bar{M}_1(w) \equiv \lambda w \pmod{p}.
    \]
\end{enumerate}
\end{cor}

\begin{proof}
    This follows from \cref{index} and \cref{commute_basis}.
\end{proof}

\subsection{The u.c.\ property}

While the following lemma appears elementary, its application in the subsequent proof is somewhat counter-intuitive. \begin{lem} \label{prime_divide}
Let $p$ be a prime number and $R$ be a positive integer. Let $N$ be a positive integer such that $N \mid pR$. If $N > R$, then $p \mid N$.
\end{lem}

\begin{proof}
We proceed by contradiction. Assume $p \nmid N$. Since $p$ is prime, this implies $\gcd(N, p) = 1$. Given that $N$ divides the product $pR$, by Euclid's Lemma, we must have $N \mid R$. This implies that $N \le R$, which contradicts the hypothesis that $N > R$. Therefore, $p \mid N$.
\end{proof}

\begin{lem} \label{contains_non_trivial}
    Let $G = A \rtimes (B \oplus C)$ be a finite group. Let $H$ be a subgroup of $G$ such that $|H| > |A||C|$ and $A \subseteq H$. Then $H$ contains an element $(1_A, b, 1_C)$ for some $b \neq 1_B$.
\end{lem}

\begin{proof}
Consider the projection map $\pi: H \to A \times C$ defined by $\pi(a, b, c) = (a, c)$. Since $|H| > |A \times C|$, the Pigeonhole Principle guarantees the existence of distinct $h_1 = (a, b_1, c)$ and $h_2 = (a, b_2, c)$ in $H$ where $b_1 \neq b_2$. Since $A \subseteq H$, the element $(a^{-1}, 0, 0)$ is in $H$, and by closure, the following element must be in $H$:$$\Biggl( (a^{-1}, 0, 0) h_1 \Biggr) \Biggl( (a^{-1}, 0, 0) h_2 \Biggr)^{-1} = (1_A, b_1, c) (1_A, b_2, c)^{-1} = (1_A, b_1 b_2^{-1}, 1_C).$$ 
\end{proof}

\begin{lem} \label{normal_subgroup_contain_Zp}
Let $p$ be a prime and $G = \mathbb{Z}_p \rtimes \left( \bigoplus_{i=1}^m \mathbb{Z}_{r_i} \right)$ be a finite group, where each $r_i \mid (p-1)$. Define $r = r_1$, and let $1_r$ be the generator of $\mathbb{Z}_{r}$, and suppose $1_r$ acts on $\mathbb{Z}_p$ via the map $1_r \cdot 1 = k \in \mathbb{Z}_p \setminus \{0\}$, where $r = \text{ord}_p(k)$. Let $N$ be a normal subgroup of $G$. If the quotient group $G/N$ contains a nilpotent subgroup with index strictly less than $r$, then $N$ must contain the subgroup $\mathbb{Z}_p \times \{0\}$.
\end{lem}
\begin{proof}
Our strategy is to show that a subgroup having a ``small index" forces $\mathbb{Z}_p$ to be contained in each lower central term of this subgroup, from which the statement follows immediately. We begin by showing that $\mathbb{Z}_p$ is contained in the whole subgroup. Let $\frac{\Gamma'}{N}$ be the nilpotent subgroup of $\frac{\Gamma}{N}$ with index $< r$. Then, we have:
\[ \frac{p r_1 r_2 \cdots r_m}{|\Gamma'|} = \frac{|\Gamma|}{|\Gamma'|} = \frac{|\Gamma /N|}{|\Gamma'/N|} < r =r_1. \]
 Hence, we have $r_1 \, r_2 \cdots r_m < p \, r_2 \cdots r_m < |\Gamma'|$.   Since $|\Gamma'|$ divides $p r_1 \cdots r_m$, \cref{prime_divide} tells us that we have $p \mid |\Gamma'|$. Therefore, there exists an element $(t,\underline{s}) \in \Gamma'$ with order $p$. It follows that $(0,\underline{0}) = (t,\underline{s})^p = (x, p\underline{s}) = (x, \underline{s})$ for some $x \in \mathbb{Z}_p$. Hence $\underline{s} = \underline{0}$, so we have $\mathbb{Z}_p = \langle (t, \underline{0}) \rangle \subseteq \Gamma'$. Also, since $p r_2 \cdots r_m < |\Gamma'|$, \cref{contains_non_trivial} tells us that $\Gamma'$ contains an element of the form $(0,\underline{u})$, where $\underline{u} \in \bigoplus_{i=1}^m \mathbb{Z}_{r_i}$ has only one non-zero entry $u$ on the $\mathbb{Z}_{r_1}$ component.
We claim that the subgroup $[\{0\} \times \langle \underline{u} \rangle, \mathbb{Z}_p \times \{0\}]$ contains $\mathbb{Z}_p \times \{0\}$. Note that, since $u \in \mathbb{Z}_r \setminus \{0\}$, we have
\begin{equation} \label{obtain_an_generator}
    u \cdot 1 = k^{u} \not \equiv 1 \pmod{p}.
\end{equation}
It follows that $[\{0\} \times \mathbb{Z}_r, \mathbb{Z}_p \times \{0\}] \ni (0,\underline{u} ) (1,0) (0,-\underline{u}) (-1,0) = (k^{u} - 1, 0)$. Hence, $[\{0\} \times \mathbb{Z}_r, \mathbb{Z}_p \times \{0\}] \supseteq \langle (k^{u} - 1, 0) \rangle = \mathbb{Z}_p \times \{0\}$.
Therefore, by induction, we can deduce that every term in the lower central series of $\Gamma'$ contains $\mathbb{Z}_p \times \{0\}$. Since the quotient $\frac{\Gamma'}{N}$ is nilpotent, $N$ must contain a term in the lower central series of $\Gamma'$. Therefore, we have $\mathbb{Z}_p \subseteq N$.
\end{proof}

Note that \eqref{obtain_an_generator} shows why \( r \) being the multiplicative order of \( k \) is important, as any non-identity element in \( \mathbb{Z}_r \) conjugated with a non-identity element in \( \mathbb{Z}_p \) gives an element that is not the multiplicative identity of \( \mathbb{Z}_p \).

The following theorem is a consequence of a strengthened version of Gromov’s Theorem by Breuillard, Green, and Tao \citep{bgt}.
 
\begin{thm}[{\citep[Theorem 4.1]{nil}}]\label{matt_lemma} Let $\varepsilon, \delta>0$. Then there exist constants $C_{\epsilon,\delta}$ and  $D_{\epsilon,\delta}$ depending only on $\epsilon,\delta$ such that the following holds. Suppose $F$ is a finite group generated by a symmetric subset $S$ containing the identity whose Cayley graph has diameter $\gamma:=\diam_S(F)$ and satisfies
 \[\gamma \geq \left( \frac{|F|}{|S|} \right)^\epsilon\]
and $\gamma \geq D_{\epsilon,\delta}$. Then
 $F$ has a normal subgroup $H$ contained in $S^{\lfloor\gamma^\delta\rfloor}$ such that $F/H$ has a nilpotent subgroup of index  at most $C_{\epsilon,\delta}$.
\end{thm}




We now characterise the groups of the form \( K(M_1, \dots, M_m) \rtimes \mathbb{Z}^m \) with the u.c.\ property.

\begin{prop} \label{char_K_Z_m}
Let $G = K(M_1, \dots, M_m) \rtimes \mathbb{Z}^m$. Suppose $G$ has uniformly almost flat coset spaces. Then every eigenvalue of every $M_i$ is a root of unity. 
\end{prop}

\begin{proof}
 We will prove the statement via contradiction. We begin by defining a sequence of quotients of $G$. Write $K = K(M_1, \dots, M_m)$. Without loss of generality, assume that $M_1$ has an eigenvalue that is not a root of unity. According to \cref{mult_char_polys}, there exists an infinite set of primes $\mathcal{P}$ such that $\mathcal{P} \cap D = \varnothing$ and for each $p \in \mathcal{P}$, the characteristic polynomial $\operatorname{char}(M_i)$ splits over $\mathbb{F}_p$ with no zero roots. For each $p \in \mathcal{P}$, let $\lambda_{p,1}$ be an eigenvalue of $\tilde{M}_1$ with the largest multiplicative order over $\mathbb{F}_p$. Let $A_p$ be a subgroup of $K$ with index $p$, as defined in \cref{index_p_subgroup_rat} using the matrices $\{M_i\}$ and the eigenvalue $\lambda_{p,1}$. Let $\lambda_{p,i}$ be the corresponding eigenvalue of $\tilde{M}_i$ associated with $\tilde{v}_p$; in particular, each $\lambda_{p,i}$ is non-zero. Let $r_{p,1} = \operatorname{ord}_p(\lambda_{p,1})$ and choose $r_{p,2}, \dots, r_{p,m} \in \{1, \dots, p-1\}$ such that $\lambda_{p,i}^{r_{p,i}} \equiv 1 \pmod{p}$. For example, for $i = 2, \dots, m$, one could take $r_{p,i} = p-1$. We will consider the following normal subgroup of $G$ and its corresponding quotient:
\[ H_p = A_p \rtimes \bigoplus_{i=1}^m r_{p,i} \mathbb{Z} \trianglelefteq G \quad \text{and} \quad \Gamma_p \coloneqq \frac{G}{H_p} \cong \mathbb{Z}/p\mathbb{Z} \rtimes \bigoplus_{i=1}^m \mathbb{Z}/r_{p,i}\mathbb{Z}. \]
Let $S$ be a generating set of $G$. Suppose $G$ has uniformly $\alpha$-almost flat coset spaces, i.e., there exist $\alpha \in (0,1]$ and $c>0$ such that
\begin{equation} \label{contradiction_uniform}
\operatorname{diam}(G/H) \ge c|G:H|^\alpha
\end{equation}
for all finite index subgroups $H$ of $G$. Fix $\epsilon < \frac{\alpha}{2}$ and $\delta = \frac{\epsilon}{m+1}$. With a slight abuse of notation, let $\bar{S}$ be the image of the generating set $S$ in $\Gamma_p$. For sufficiently large $p$, we have $\operatorname{diam}(\Gamma_p) \ge c|\Gamma_p|^\alpha \ge |\Gamma_p|^{\alpha/2}$. We will now apply \cref{matt_lemma}. Again, by taking $p$ large enough, we have $\operatorname{diam}(\Gamma_p) \ge D_{\epsilon,\delta}$, and so there exists $N_p \trianglelefteq \Gamma_p$ such that:
\begin{enumerate}[(i)]
    \item $N_p \subseteq \bar{S}^{\lfloor \operatorname{diam}(\Gamma_p)^\delta \rfloor}$;
    \item $\frac{\Gamma_p}{N_p}$ has a nilpotent subgroup with index bounded by $C_{\epsilon,\delta}$.
\end{enumerate}
By \cref{diverges_mul_order_rat}, for sufficiently large $p$, we have $r_{p,1} > C_{\epsilon,\delta}$. Therefore, by \cref{normal_subgroup_contain_Zp}, we have $N_p \supseteq \mathbb{Z}/p\mathbb{Z} \times \{0\}$. Moreover, we trivially have $\operatorname{diam}(\Gamma_p) \le |G:H_p| = p \cdot r_{p,1} \cdots r_{p,m} \le p^{m+1}$. Hence, $$\operatorname{diam}(\Gamma_p)^\delta \le p^{(m+1)(\frac{\epsilon}{m+1})} = p^\epsilon.$$ By the choice of $N_p$, we have $\mathbb{Z}/p\mathbb{Z} \times \{0\} \subseteq N_p \subseteq \bar{S}^{\lfloor \operatorname{diam}(\Gamma_p)^\delta \rfloor} \subseteq \bar{S}^{p^\epsilon}$. Finally, taking $H' = A_p \rtimes \mathbb{Z}^m$, we see that $S^{p^\epsilon} H' = G$. Therefore, $$\operatorname{diam}(G/H') \le p^\epsilon = |G:H'|^\epsilon < c|G:H'|^\alpha$$ for large $p$, which contradicts \eqref{contradiction_uniform}.
\end{proof}

\subsection{Equivalent conditions for $K(M_1, \dots, M_m) \rtimes \mathbb{Z}^m$ with u.c.\ property}
Our next aim is to prove the following:

\begin{prop}\label{K_by_zm}

Let $G = K(M_1, \dots, M_m) \rtimes \mathbb{Z}^m$, where each $M_i \in \mathrm{GL}_n(\mathbb{Q})$. Then:
\begin{enumerate}[(i)]

\item If all eigenvalues of each $M_i$ are $1$, then $G$ is nilpotent of class at most $n$.

\item If all eigenvalues of each $M_i$ are roots of unity, then $G$ is virtually nilpotent.

\end{enumerate}

\end{prop}

One way to prove this is by combining \cref{Finite_generation_criteria} and \citep[Lemma 9.18]{Guo_Tointon}. We will provide a straightforward, stand-alone proof. Let $G$ be a group. For $x_1, \dots, x_n \in G$, the left-normed $n$-fold commutator is defined by:

$$[x_1, \dots, x_n] := [[x_1, \dots, x_{n-1}], x_n] \text{ for } n \ge 2.$$

\begin{lem} \label{thm:lower_central_series}

Let $G = A \rtimes H$ be a semidirect product where both $A$ and $H$ are abelian. For all $i \ge 1$, the $(i+1)$-th term of the lower central series of $G$ is given by $G_{i+1} = \langle [a, h_1, \dots, h_i] \mid a \in A,\ h_j \in H \rangle \subseteq A$.

\end{lem}

\begin{proof}

For convenience of notation, we use multiplicative notation for $G, A,$ and $H$ throughout this proof; moreover, we may not explicitly introduce a variable when it is clear where it belongs. We proceed by induction on $i$. For the base case ($i = 1$), $G_2 = [G, G]$ is the normal closure of commutators of the generators of $G$. Since $G = \langle A \cup H \rangle$ with $A$ and $H$ abelian, $G_2 = \langle [a, h] \mid a \in A,\ h \in H \rangle^G$. For $g = a'k \in G$, we have $[a, h]^g = [a, h]^k$ since $A$ is abelian and $[a, h] \in A$. Using $[a, h]^k = [a^k, h^k]$ and $h^k = h$, we obtain $[a, h]^k = [a^k, h] \in [A, H]$. Hence the generating set $\{[a, h] \mid a \in A,\ h \in H\}$ is closed under taking $G$-conjugates, and therefore $G_2 = \langle [a, h] \mid a \in A,\ h \in H \rangle$.

For the inductive step, assume that $G_{i+1} = \langle [a, h_1, \dots, h_i] \mid a \in A,\ h_j \in H \rangle \subseteq A$. Then $G_{i+2} = [G_{i+1}, G] = \langle [x, g] \mid x \in G_{i+1},\ g \in A \cup H \rangle^G$. Since $A$ is abelian, $[x, a'] = 1$ for all $a' \in A$, and hence $G_{i+2} = \langle [x, h] \mid x \in G_{i+1},\ h \in H \rangle^G$. Moreover, given $k \in H$, $[x, h]^k = [x^k, h^k] = [x^k, h]$, so the generating set $\{[x, h] \mid x \in G_{i+1},\ h \in H\}$ is closed under taking $G$-conjugates. Hence, we obtain:

$$G_{i+2} = \langle [[a, h_1, \dots, h_i], h_{i+1}] \mid a \in A,\ h_j \in H \rangle = \langle [a, h_1, \dots, h_{i+1}] \mid a \in A,\ h_j \in H \rangle.$$

This completes the induction.

\end{proof}

\begin{lem} \label{product_of_n_nilpootent_matrices}

Let $Q_1, \dots, Q_n$ be a set of pairwise commuting matrices with all eigenvalues $0$. Then $\prod_{i=1}^n Q_i = 0$.

\end{lem}

\begin{proof}

The result follows from this standard fact: Let $S$ and $T$ be two commuting linear maps on $V$. If all eigenvalues of $S$ are $0$ and $T$ is a non-zero linear map, then $\ker(T) \subsetneqq \ker(ST)$ (unless $\ker(T)=V$).

\end{proof}
\begin{proof} [Proof of \cref{K_by_zm}]
    We prove (i); (ii) follows immediately. Take $h \in \mathbb{Z}^m$. The action of $h$ on $a \in A$ is defined by $a^h = M_h a$ for some matrix $M_h$ with all eigenvalues equal to $1$. 

In additive notation, the commutator $[a, h] = a^h - a$ translates to the linear map $(M_h - I)a$. By induction, the $(i+1)$-fold commutator from \cref{thm:lower_central_series} corresponds to the product of these matrices:
\begin{equation} \label{product_commutator}
[a, h_1, h_2, \dots, h_i] = (M_{h_i} - I)(M_{h_{i-1}} - I) \dots (M_{h_1} - I)a.
\end{equation}
Part (i) of the statement follows from \eqref{product_commutator} and \cref{product_of_n_nilpootent_matrices}.
\end{proof}

\begin{cor} \label{A_by_Z_finite_rank}
The following statements about $G = K(M_1, \dots, M_m) \rtimes \mathbb{Z}^m$ are equivalent:
\begin{enumerate}[(i)]
    \item all eigenvalues of each $M_i$ are roots of unity;
    \item the group $G$ is virtually nilpotent;
    \item the group $G$ has polynomial growth;
    \item the group $G$ has uniformly almost flat coset spaces.
\end{enumerate}
\end{cor}

\begin{proof}
The implication (i) $\Rightarrow$ (ii) is proved in \cref{K_by_zm}. The equivalence (ii) $\Leftrightarrow$ (iii) follows from Gromov's theorem. The implication (iii) $\Rightarrow$ (iv) is proved in \cref{poly_growth_implies_uniform_D_a}. Finally, the implication (iv) $\Rightarrow$ (i) is proved in \cref{char_K_Z_m}.
\end{proof}

\subsection{The Abelian-by-Nilpotent Case with a Split Extension}
 The conclusion of this subsection follow from the results cited in the proof of \cref{uc_implies_polycyclic}, we include it to provide an intuition for how the just-non-polycyclic (JNP) property simplifies the group structure. Recall that a group $G$ is a \textit{JNP group} if $G$ is not polycyclic, but every proper quotient of $G$ is polycyclic. We will explore more the application of JNP groups in \cref{reduction_non_polycyclic_structure}.

\begin{lem}
Suppose $G$ is a JNP group that is abelian-by-nilpotent. Then $\text{Fit}(G)$ is abelian and $G/\text{Fit}(G)$ is nilpotent.
\end{lem}

\begin{proof}
Since $G$ is abelian-by-nilpotent, $\text{Fit}(G)$ is nilpotent by a result of Hall \citep{peter_hall_finiteness_certain_soluble_groups}. Furthermore, the fact that $G$ is an extension of an abelian group by a nilpotent group implies that $G/\text{Fit}(G)$ is nilpotent. As $G$ is JNP, it is by definition not polycyclic; therefore, $\text{Fit}(G)$ must be infinitely generated. We observe that the abelianization of an infinitely generated nilpotent group is itself infinitely generated. Thus, the quotient $\text{Fit}(G)/[\text{Fit}(G), \text{Fit}(G)]$ is an infinitely generated abelian group. Since $G$ is JNP, every proper quotient of $G$ is polycyclic. If the derived subgroup $[\text{Fit}(G), \text{Fit}(G)]$ were non-trivial, then the quotient $G/[\text{Fit}(G), \text{Fit}(G)]$ would be polycyclic. However, this quotient contains $\text{Fit}(G)/[\text{Fit}(G), \text{Fit}(G)]$, which is infinitely generated, a contradiction. Consequently, we must have $[\text{Fit}(G), \text{Fit}(G)] = 1$, which proves that $\text{Fit}(G)$ is abelian.
\end{proof}

\begin{lem} \label{commute_implies_root_of_unity}
Let $A,B\in\GL(n,\C)$ and $C=[A,B]$. If $AC=CA$ and $BC=CB$, then every eigenvalue of $C$ is a root of unity.
\end{lem}

\begin{proof}
For an eigenvalue $\lambda$ of $C$, the eigenspace $E_\lambda=\ker(C-\lambda I)$ is $A$- and $B$-invariant. Hence on $E_\lambda$ we have
$[A|_{E_\lambda},B|_{E_\lambda}]=C|_{E_\lambda}=\lambda I$.
Taking determinants on both sides gives $\lambda^{\dim E_\lambda}=1$, so $\lambda$ is a root of unity.
\end{proof}

\begin{lem} [Roots of unity in the center of nilpotent groups ]\label{non_abelian_implies_center_has_one_with_root_of_unity}
    Suppose $H$ is a nilpotent group and $\phi : H \to \mathrm{GL}(n,\mathbb{C})$ is a homomorphism.  
    If $H$ has nilpotency class $c>1$, then there exists $x \in H_c \setminus \{1\}$ such that all eigenvalues of $\phi(x)$ are roots of unity.
\end{lem}

\begin{proof}
Suppose $H$ has nilpotency class $c>1$. Then there exist $g \in H_{c-1}$ and $h \in H$ such that $[g,h] \neq 1$. By definition of the lower central series, $[g,h] \in H_c \subseteq Z(H)$. Since $\phi$ is a homomorphism, $\phi([g,h]) = [\phi(g),\phi(h)]$, and $\phi([g,h])$ commutes with both $\phi(g)$ and $\phi(h)$. It follows from \cref{commute_implies_root_of_unity} that all eigenvalues of $\phi([g,h])$ are roots of unity.
\end{proof}

\begin{prop}
Let $G = Fit(G) \rtimes \Gamma$ be a finitely generated JNP group, where $Fit(G)$ is an abelian group of rank $n$ and $\Gamma$ is a torsion-free nilpotent group. For any $x \in Z(\Gamma)  \setminus \{1\}$, let $M_x$ be the matrix representing the action of $x$ on $Fit(G)$. Then $M_x$ cannot have all eigenvalues as roots of unity. It follows that $G/Fit(G)$ is in fact abelian, and $G$ does not have the u.c.\ property.
\end{prop}

\begin{proof}
 Let $x \in Z(\Gamma)  \setminus \{1\}$. Suppose that $M_x$ is a matrix with all eigenvalues being roots of unity. Let $L \in \mathbb{N}$ be such that $M_x^L$ has all eigenvalues equal to $1$. Following the construction in the proof of \cref{K_by_zm}, we see that the subgroup $\langle Fit(G), x^L \rangle$ is a normal nilpotent group of $G$. This contradicts the property that $Fit(G)$ contains all normal nilpotent subgroups of $G$. Hence, $M_x$ cannot have all eigenvalues being roots of unity. 

By \cref{non_abelian_implies_center_has_one_with_root_of_unity}, it follows that $\Gamma$ is in fact abelian. Consequently, $G$ does not have the u.c.\ property by \cref{A_by_Z_finite_rank}.
\end{proof}

\section{Infinite subgroup rank}

\subsection{Existence of a quotient with infinitely many $p$-torsion elements}
Following \citep{peter_hall_finiteness_certain_soluble_groups}, let $\mathfrak{B}$ denote the class of countable abelian groups $B$ that embed in a finitely generated group $G$ with $G/B$ polycyclic. The conclusions of this subsection also follow from the results cited in the proof of \cref{uc_implies_polycyclic}. Nevertheless, we include this discussion to make explicit the role of Hall’s work in our argument, specifically in reducing the case of infinite subgroup rank to the existence of infinitely many torsion elements of prime order.

We thank Peter Kropholler for pointing out Hall’s work \citep{peter_hall_finiteness_certain_soluble_groups}. While the statement of the relevant result mentioned below in the original paper may not be immediately transparent, the necessary definitions are given in \citep[Section~2]{peter_hall_finiteness_certain_soluble_groups} , in particular in equations (1), (4), (7), and (9).

\begin{thm}[{\citealp[Lemma~5.2]{peter_hall_finiteness_certain_soluble_groups}}] \label{lem:hall_result}
    Let $A \in \mathfrak{B}$. Then there exists a finite set of primes $P$ and a free abelian subgroup $B \le A$ such that the quotient group $A/B$ is a $P$-torsion group. 
\end{thm}

\begin{cor}  \label{contain_z_infinite_torsion}
    Let $B, P$ be defined as in \cref{lem:hall_result}. Let $r = \operatorname{rank}(B) \in \mathbb{N} \cup \{\infty\}$. Then for a prime $q \notin P$, we have $A/qA \cong (\mathbb{Z}/q\mathbb{Z})^r$.
\end{cor}

Note that in the case where $A$ is a torsion group and $q \notin P$, we have $A/qA \cong (\mathbb{Z}/q\mathbb{Z})^0 \cong \{0\}$, which is consistent with the result. Furthermore, \cref{contain_z_infinite_torsion} actually provides another way of proving \cref{index}.
\begin{proof}
    Let $q \notin P$. Since $A/B$ is $P$-torsion and $q \notin P$, the group $A/B$ has no $q$-torsion; thus, multiplication by $q$ on $A/B$ is bijective. Define the map:
    \[
    \phi : B/qB \to A/qA, \quad b+qB \mapsto b+qA.
    \]
    Since $B \subseteq A$, the map $\phi$ is well-defined.   
    We will show $\phi$ is bijective, starting with injectivity. If $\phi(b+qB) = b+qA = 0$, then $b \in B \cap qA$. This implies $b = qa$ for some $a \in A$. In the quotient $A/B$, we have    
  $  q(a+B) = b+B = 0.$    
    Since multiplication by $q$ on $A/B$ is injective, we have $a+B = 0$, which means $a \in B$. Thus $b = qa \in qB$, proving that $\phi$ is injective.    
    For surjectivity, let $a+qA \in A/qA$. Since multiplication by $q$ on $A/B$ is surjective, there exists $c \in A$ such that $a+B = q(c+B)$. This implies $b = a - qc \in B$. Since $a \equiv b \pmod{qA}$, we have    
   $ \phi(b+qB) = b+qA = a+qA.$    
    Thus $\phi$ is surjective. Hence $A/qA \cong B/qB$. Given $B \cong \mathbb{Z}^r$, we obtain $B/qB \cong (\mathbb{Z}/q\mathbb{Z})^r$, and therefore $A/qA \cong (\mathbb{Z}/q\mathbb{Z})^r$.
\end{proof}

\begin{prop}[Existence of a quotient with infinitely many $p$-torsion elements]  \label{quotient_with_inf_tor}
Let $A \in \mathfrak{B}$ such that $A$ has infinite subgroup rank. Then there exists a prime $p$ such that $A/pA$ is an infinite elementary abelian $p$-group. In the case where $A$ contains a copy of $\mathbb{Z}^{\mathbb{N}}$, the statement holds for all sufficiently large $p$. 
\end{prop}
\begin{proof}According to \citep[Lemma 8]{peter_hall_finiteness_certain_soluble_groups}, we may write $A = A_{t} \oplus A_{tf}$, where $A_t$ is the torsion subgroup with bounded exponent and $A_{tf}$ is the torsion-free part. We first consider the case where $A_t$ is infinite. Since $A_t$ has bounded exponent, there exists a finite set of primes $P_t$ such that:$$A_t = \bigoplus_{p \in P_t} A_t(p)$$where each $A_t(p)$ is an abelian $p$-group. Since $A_t$ is infinite and has bounded exponent, there exists $q \in P_t$ and $k \in \mathbb{N}$ such that $A_t(q)$ is infinite and $A_t(q) = \{a \in A \mid q^k a = 0\}$. In particular, $A_t(q)$ has infinitely many elements of order $q$. Hence $A_t(q) / qA_t(q)$ is an infinite elementary abelian $q$-group. Let $C = \bigoplus_{p \in P_t \setminus \{q\}} A_t(p)$. Then we have:$$\frac{A}{qA} \cong \frac{A_t(q) \oplus C \oplus A_{tf}}{qA_t(q) \oplus qC \oplus qA_{tf}} \cong \frac{A_t(q) \oplus A_{tf}}{qA_t(q) \oplus qA_{tf}}$$which is an infinite elementary abelian $q$-group.

Now, suppose $A_t$ is finite. Since $A$ has infinite subgroup rank,  $A_{tf}$ must contain a subgroup isomorphic to $\mathbb{Z}^\mathbb{N}$. This case follows immediately from \cref{contain_z_infinite_torsion}.\end{proof}

\subsection{The u.c property}
Let $G = A \rtimes \Gamma$ be a finitely generated group, where $\Gamma$ is polycyclic and $A$ is abelian of infinite subgroup rank. Our aim is to construct a sequence of finite coset spaces that to show $G$ does not have the u.c.\ property \eqref{diam_lower_bound}. The following lemma will help provide an upper bound on the diameter of this sequence.

\begin{prop}[Polynomial word length in conjugate-generated abelian subgroups ] \label{diam_in_semi}
Let $G$ be a finitely generated group with a torsion abelian normal subgroup $A$ such that $G/A = \langle t_1 A, \dots, t_m A \rangle$. Let $R \subset A$ be a finite set, $M$ be the exponent of $\langle R \rangle$, and $S = R \cup \{t_1, \dots, t_m\}$. For $l_1, \dots, l_m \in \mathbb{N}$, we define $A^c_{l_1, \dots, l_m} \le A$ as the subgroup generated by the action of the elements $g = t_1^{k_1} \cdots t_m^{k_m}$ on $R$:
\begin{equation} \label{eq:subgroup_A}
A^c_{l_1, \dots, l_m} = \langle (t_1^{k_1} \cdots t_m^{k_m}) \cdot r \mid |k_i| \leq l_i,\; r \in R \rangle.
\end{equation}
Let $l = \sum_{i=1}^m l_i$. Denote by $d(A^c_{l_1, \dots, l_m})$ the minimal integer such that $A^c_{l_1, \dots, l_m} \subseteq S^d$. Then for a fixed $R$, $d(A^c_{l_1, \dots, l_m})$ is polynomially bounded in $l$.


\end{prop}

Recall that a finite abelian group with generating set $S$ and exponent $m$ has its diameter upper bounded by $|S|m$; this is the key idea for our proof.

\begin{proof}
For each $i$, we have $S^{l_i} \supseteq \{ t_i^{k_i} \mid k_i \in \{1, \dots, l_i\} \}$. Hence, for any $k_i \in \{1, \dots, l_i\}$ and $r \in R$,$$(t_1^{k_1} \cdots t_m^{k_m}) \cdot r = (t_1^{k_1} \cdots t_m^{k_m}) r (t_1^{k_1} \cdots t_m^{k_m})^{-1} \in S^{2l+1}.$$Therefore, every generator in the generating set of $A^c_{l_1, \dots, l_m}$ defined in \eqref{eq:subgroup_A} is contained in $S^{2l+1}$. It is clear that the exponent of $A^c_{l_1, \dots, l_m}$ is the same as the exponent of $\langle R \rangle$, which is $M$. We now count the generators in \eqref{eq:subgroup_A}. Writing $\mathbf{k} = (k_1, \dots, k_m)$, there are at most $|R| \prod_{i=1}^m l_i$ choices for the pair $(r, \mathbf{k})$ with $r \in R$ and $k_i \in \{1, \dots, l_i\}$. This bounds the number of generators of $A^c_{l_1, \dots, l_m}$ defined in \eqref{eq:subgroup_A}.It follows that:$$S^{(2l+1) M |R| \prod_{i=1}^m l_i} \supseteq A^c_{l_1, \dots, l_m}.$$Thus,$$d(A^c_{l_1, \dots, l_m}) \le (2l+1) M |R| \prod_{i=1}^m l_i \le (2l+1) M |R| l^m,$$as required.
\end{proof}

The following result will help us compute a lower bound of the size of the finite coset space.

\begin{lem} \label{once_stop_then_stablise}
Let $G$ be a group and $N \trianglelefteq G$. Let $f \in \operatorname{Aut}(G)$ be such that $f(N)=N$. Let $R \subseteq G$ be a finite set. For $k \in N \cup \{\infty\}$, define
\[ H_k = \langle f^i(R) \cup N \mid |i| \leq k \rangle \]
If $H_L = H_{L+1}$ for some $L \geq 0$, then $H_L = H_\infty$.
\end{lem}

\begin{proof}
    Since $N \trianglelefteq G$ and $f(N) = N$, the automorphism $f$ induces a well-defined automorphism $\bar{f}$ on the quotient group $\bar{G} = G/N$, defined by $\bar{f}(gN) = f(g)N$. Let $\bar{R} = \{rN \mid r \in R\}$. The subgroups $H_k$ correspond to the following subgroups in the quotient:    \[ \bar{H}_k = \langle \bar{f}^i(\bar{R}) \mid |i| \leq k \rangle \leq \bar{G} \]
    The hypothesis $H_L = H_{L+1}$ implies $\bar{H}_L = \bar{H}_{L+1}$ in $\bar{G}$.     The subgroup $\bar{H}_{L+1}$ is generated by $\bar{H}_L$ along with the elements $\bar{f}^{L+1}(\bar{R})$ and $\bar{f}^{-(L+1)}(\bar{R})$. Since $\bar{H}_L = \bar{H}_{L+1}$, it follows that:
    \[ \bar{f}^{L+1}(\bar{R}) \subseteq \bar{H}_L \quad \text{and} \quad \bar{f}^{-(L+1)}(\bar{R}) \subseteq \bar{H}_L \]

    We check if $\bar{H}_L$ is invariant under the automorphism $\bar{f}$. Consider the image of the generating set:
    \[ \bar{f}(\bar{H}_L) = \langle \bar{f}^{-L+1}(\bar{R}), \dots, \bar{f}^L(\bar{R}), \bar{f}^{L+1}(\bar{R}) \rangle \]
    By definition, all generators $\bar{f}^j(\bar{R})$ for $-L+1 \leq j \leq L$ are in $\bar{H}_L$. The ``new" generator $\bar{f}^{L+1}(\bar{R})$ is also in $\bar{H}_L$ by our previous observation. Thus, $\bar{f}(\bar{H}_L) \subseteq \bar{H}_L$.    Similarly, consider the inverse $\bar{f}^{-1}$:
    \[ \bar{f}^{-1}(\bar{H}_L) = \langle \bar{f}^{-L-1}(\bar{R}), \dots, \bar{f}^{L-1}(\bar{R}) \rangle \]
    The ``boundary" generator $\bar{f}^{-(L+1)}(\bar{R})$ is in $\bar{H}_L$; thus, $\bar{f}^{-1}(\bar{H}_L) \subseteq \bar{H}_L$. Since $\bar{f}^{-1}(\bar{H}_L) \subseteq \bar{H}_L$ implies $\bar{H}_L \subseteq \bar{f}(\bar{H}_L)$, we conclude that $\bar{f}(\bar{H}_L) = \bar{H}_L$.     Because $\bar{H}_L$ is invariant under both $\bar{f}$ and its inverse, it must contain all powers $\bar{f}^i(\bar{R})$ for all $i \in \mathbb{Z}$. Therefore, $\bar{H}_L = \bar{H}_\infty$. Since $H_i \supseteq N$ for all $i \in \mathbb{N} \cup \{\infty\}$ by definition, we obtain $H_L = H_\infty$.
\end{proof}

\begin{lem} [Bound on the size of a set of complete coset representations.] \label{inf_g_abelian_bdd_generated_coset}
Let $G$ be a finitely generated group and $A \trianglelefteq G$ such that $Q = G/A$ is boundedly generated by $\{t_1 A, \dots, t_m A\}$ and $A$ is normally generated by a finite subset $R \subset A$. Let $N$ be a finite index subgroup of $A$.
Then:
\begin{enumerate}
    \item There exist $l_1, \dots, l_m$ such that 
    \begin{equation} \label{eq:subgroup_A_lemma}
    A^c_{l_1, \dots, l_m} N = A
    \end{equation}
    \item Suppose each $l_i$ is minimal, in the sense that if we replace any $l_i$ with any $t < l_i$, we have 
    \[ A^c_{l_1, \dots, t, \dots, l_m} N \neq A^c_{l_1, \dots, l_i, \dots, l_m} N, \]
    then  $[A:N]$ has size at least $p^{\sum l_i}$, where $p$ is the minimum prime divisor of the index $[A:N]$.
\end{enumerate}
\end{lem}

\begin{proof}
We start with (i). Let \( C \) be a complete set of coset representatives of \( N \) in \( A \). Since \( N \) is a finite index subgroup of \( A \), \( C \) is finite. Since \( A \) is normally generated by \( R \), every element of \( C \) is in \( A^c_{l_1,\dots,l_m} \) for sufficiently large \( l_i \), and from this the statement follows immediately.

We next prove (ii). By \cref{once_stop_then_stablise}, for every non-negative \( t < l_i \), \( A^c_{0,\dots,0,t,l_{i+1},\dots,l_m}N \) is a proper subgroup of \( A^c_{0,\dots,0,t+1,l_{i+1},\dots,l_m}N \). The statement follows by induction and the Correspondence Theorem for groups.
\end{proof}
Given a semidirect product, the following lemma will help us pass from a normal filtration of the whole semidirect product to a normal filtration of the kernel. Later in the proof of \cref{inf_g_abelian_bdd_generated_coset}, we will use this normal filtration of the kernel to construct a filtration of the whole group.
\begin{lem} \label{filtration_pass_normal_subgroup_swapped}
Suppose $G = N \rtimes H$ admits a normal filtration $(K_i)_{i \in \mathbb{N}}$  with trivial intersection. Then $N$ admits a normal filtration $(N_i)_{i \in \mathbb{N}}$ with trivial intersection such that each $N_i \trianglelefteq G$.
\end{lem}

\begin{proof}
Let $(K_i)_{i \in \mathbb{N}}$ be a normal filtration of $G$ with $\bigcap_{i \in \mathbb{N}} K_i = \{1\}$. Define $N_i := K_i \cap N$. Then each $N_i \trianglelefteq G$ since both $K_i$ and $N$ are normal in $G$.

Moreover, $[N : N_i] \le [G : K_i] < \infty$, so each $N_i$ has finite index in $N$. Finally,
\[
\bigcap_{i \in \mathbb{N}} N_i = \bigcap_{i \in \mathbb{N}} (K_i \cap N)
= \left( \bigcap_{i \in \mathbb{N}} K_i \right) \cap N
= \{1\}.
\]

Refining $(N_i)$ if necessary, we may assume it is strictly decreasing. Hence, $(N_i)_{i \in \mathbb{N}}$ is the desired normal filtration of $N$.
\end{proof}
We now combine the two bounds derived earlier to construct a sequence of coset spaces that violate the u.c.\ property.
\begin{cor} \label{inf_g_abelian_u_c}
Let $G = A \rtimes \Gamma$ be a finitely generated group such that:
\begin{enumerate}[(i)]
    \item $G$ admits a normal filtration with trivial intersection;
    \item $\Gamma$ is boundedly generated;
    \item $A$ is an infinite torsion abelian group of bounded exponent.
\end{enumerate}
Then $G$ does not have uniformly almost flat coset spaces.
\end{cor}

\begin{proof}
Let $R \subset A$ be finite and $t_1, \dots, t_m \in \Gamma$ such that $\Gamma$ is boundedly generated by $\{t_1, \dots, t_m\}$. Then
\[
S := R \cup \{t_1^{\pm 1}, \dots, t_m^{\pm 1}\}
\]
is a symmetric generating set of $G$. By \cref{filtration_pass_normal_subgroup_swapped}, $A$ admits a filtration $(A_i)_{i \in \mathbb{N}}$ with each $A_i \trianglelefteq G$. Setting $K_i := A_i \rtimes \Gamma$, we obtain a filtration of $G$.  Using the setup in \cref{inf_g_abelian_u_c}, for each $i$, let $l_{i,1}, \dots, l_{i,m}$ be minimal integers such that $A^c_{l_{i,1}, \dots, l_{i,m}} A_i = A$, and set $l(i) = \sum_{j=1}^m l_{i,j}$. Then
\[
p^{l(i)} \le [A : A_i] = [G : K_i].
\]
On the other hand, \cref{diam_in_semi} shows that $\mathrm{diam}_S(G/K_i)$ is bounded above by a polynomial in $l(i)$. Since $l(i)$ is unbounded, it follows that $G$ does not have uniformly almost flat coset spaces.
\end{proof}

\begin{cor}\label{infinite_subgroup_rank}
Suppose \( G = A \rtimes \Gamma \) is a finitely generated group where \( \Gamma \) is polycyclic and \( A \) is abelian with infinite subgroup rank. Then \( G \) does not have uniformly almost flat coset spaces.
\end{cor}

\begin{proof}
   By the Roseblade--Jategaonkar result \citep{Roseblade, JAT}, finitely generated abelian-by-polycyclic groups are residually finite.  The statement follows from \cref{quotient_with_inf_tor}   and \cref{inf_g_abelian_u_c}
\end{proof}



\section{Soluble groups: conclusion of the proof}

As outlined in \cref{into}, our strategy is to reduce to the case of abelian-by-abelian groups with a split extension. This is achieved in the following proposition.

\begin{prop} \label{reduction_non_polycyclic_structure}
Let $G$ be a finitely generated abelian-by-nilpotent group that is not polycyclic. Then there exists $\Gamma \in \mathcal{C}(G)$ such that $\Gamma \cong A \rtimes \mathbb{Z}^m$, where $A$ is an abelian group that is not finitely generated.
\end{prop}

In the following proof, we will make use of the terms \textit{just non-polycyclic (JNP)} group and \textit{primitive just non-polycyclic (PJNP)} group. Recall a group $G$ is a JNP group if $G$ is not polycyclic, but every proper quotient of $G$ is polycyclic. Such groups have been studied extensively by Robinson and Wilson \citep{Robin_Wilson_jnp} using techniques from group cohomology \citep{MR679154} and the theory of modules over polycyclic group rings \citep{peter_hall_finiteness_certain_soluble_groups}. A PJNP group is a JNP group satisfying additional structural conditions (see \citep[p.208]{Robin_Wilson_jnp}); we shall primarily utilize the structural properties of PJNP groups.

\begin{proof}
    Since $G$ is not polycyclic, there exists  $Q \in \mathcal{C}(G)$ that is PJNP by \citep[p.208, (4.4)]{Robin_Wilson_jnp}. The Fitting subgroup $\operatorname{Fit}(Q)$ is abelian by \citep[p.197, (2.3)]{Robin_Wilson_jnp} and cannot be finitely generated since $Q$ is not polycyclic.    We first show that the Fitting quotient of $Q$ is nilpotent. Since $G$ is abelian-by-nilpotent, $Q$ inherits this property by \cref{class_closure_abelian_by_nilpotent}; thus, there exists an abelian normal subgroup $K \trianglelefteq Q$ such that $Q/K$ is nilpotent. In particular, this implies that $K \subseteq \operatorname{Fit}(Q)$. By the Third Isomorphism Theorem, $Q/\operatorname{Fit}(Q) \cong (Q/K)/(\operatorname{Fit}(Q)/K)$, which implies that $Q/\operatorname{Fit}(Q)$ is a quotient of a nilpotent group and is therefore nilpotent.     By \citep[p.210, (2.4) \& (5.1)]{Robin_Wilson_jnp}, $Q$ has a free abelian subgroup $X \cong \mathbb{Z}^m$ such that $\operatorname{Fit}(Q)X$ is a finite-index subgroup of $Q$ and $\operatorname{Fit}(Q) \cap X = \{1\}$. We are done by \cref{finite_index_quotient}.
\end{proof}
\begin{rem}
The following quick example demonstrates how passing from $G$ to $\mathcal{C}(G)$ simplifies the group we are working with in the case where we have infinite subgroup rank. Consider the group $G = K \rtimes H$, where $K = \bigoplus_{i \in \mathbb{Z}} \mathbb{Z}_3$ and $H$ is the discrete Heisenberg group:
\[
H = \langle x, y \mid [[x, y], x] = [[x, y], y] = 1 \rangle.
\]
The action of $H$ on the basis $\{e_i\}_{i \in \mathbb{Z}}$ of $K$ is given by $x \cdot e_i = e_{i+1}$ and $y \cdot e_i = 2^i e_i$. It follows that the commutator $z = [x, y]$ acts as the uniform scalar $z \cdot e_i = 2e_i$.

It is clear that $G_0 = K \rtimes H^2$ is a finite-index subgroup of $G$. By quotienting $G_0$ by the normal closure $\langle \langle z^2 \rangle \rangle_{G_0}$, we obtain a metabelian group containing the infinite-rank subgroup $\bigoplus_{i \in \mathbb{Z}} \mathbb{Z}_3$.
\end{rem}

\begin{rem}
    Note that \cref{reduction_non_polycyclic_structure} does not generalise to abelian-by-polycyclic groups; see \citep[p.~227, (7.9)]{Robin_Wilson_jnp}.
\end{rem}

\begin{cor} \label{abelin_by_nil_case}
Let $G$ be a finitely generated abelian-by-nilpotent group that is not polycyclic. Then $G$ does not have the u.c.\ property.
\end{cor}

\begin{proof}
By \cref{uniform_D_alpha_passes_finte_index_subgroup} and \cref{reduction_non_polycyclic_structure}, it suffices to consider the case where $G$ is abelian-by-abelian with split extension, which can be handled by  \cref{A_by_Z_finite_rank} and \cref{infinite_subgroup_rank}.
\end{proof}

\begin{thm}\label{uc_implies_polycyclic}
A finitely generated soluble group with uniformly almost flat coset spaces is polycyclic.
\end{thm}

\begin{proof}
We argue by induction on the derived length $d$ of $G$. The case $d=1$ is clear, since a finitely generated abelian group is polycyclic. Assume the result holds for derived length $\le d$, and let $G$ have derived length $d+1$ with uniformly almost flat coset spaces. Then $G^{(d)}$ is abelian and $G/G^{(d)}$ is finitely generated soluble of derived length $d$ with uniformly almost flat coset spaces. By the induction hypothesis, $G/G^{(d)}$ is polycyclic. According to Theorem \cref{polycyclic_by_finite_group_uda_iff_vn} , $G/G^{(d)}$ is virtually nilpotent.Suppose for contradiction that $G$ is not polycyclic. Then $G^{(d)}$ is not finitely generated.  Furthermore, there exists a finite-index subgroup $H \le G$ with $G^{(d)} \le H$ such that $H/G^{(d)}$ is nilpotent. Thus $H$ is abelian-by-nilpotent but not polycyclic, so by \cref{abelin_by_nil_case}, $H$ does not have uniformly almost flat coset spaces. We are done by \cref{uniform_D_alpha_passes_finte_index_subgroup}.
\end{proof}

\begin{proof}[Proof of \cref{main_thm}]
The statement  follows immediately from \cref{uc_implies_polycyclic} and \cref{polycyclic_by_finite_group_uda_iff_vn}.
\end{proof}

\section{Uniformly almost flatness in finitely generated groups} \label{last_section}

\subsection{Groups with the u.q. property virtually have only soluble finite quotients}


The aim of this subsection is to show that groups with the u.q. property have a finite-index subgroup whose finite quotients are all soluble. This will allow us to deduce \cref{main_thm_residually_soluble} and \cref{main_thm_linear}.  We begin by recalling some standard results on minimal normal subgroups and include their proofs for the convenience of the reader.

\begin{lem}\label{min_group}
Let \(G\) be a group with a minimal normal subgroup $M$. 
   If $L \trianglelefteq G$, then either $M \leqslant L$ or $M \cap L = 1$, in which case $\langle M, L \rangle = M \times L$.


\end{lem}
 
\begin{proof}
Suppose that $G$ has a minimal normal subgroup $M$, let \(L\unlhd G\). Since \(M\cap L\unlhd G\) and \(M\cap L\leq M\), the minimality of \(M\) implies that either \(M\cap L=1\) or \(M\leq L\). In the former case, \([M,L]\leq M\cap L=1\), so \(M\) and \(L\) centralize one another. Hence \(\langle M,L\rangle=ML=M\times L\).



\end{proof}


A group is called \textit{characteristically simple} if it has no proper nontrivial characteristic subgroup. It is clear that a minimal normal subgroup must be characteristically simple.

\begin{lemma}\label{lem:char_simple_direct_power}
Let \(F\) be a finite characteristically simple group. Then
\[
F \cong T^n
\]
for some finite simple group \(T\) and some \(n \geq 1\).
\end{lemma}
\begin{proof}
Let $T$ be a minimal normal subgroup of $F$. For any automorphism $\alpha \in \operatorname{Aut}(F)$, $T^\alpha$ is also a minimal normal subgroup of $F$. Let $\{T_1, \ldots, T_m\}$ be a maximal collection of distinct automorphic images of $T$ such that their product is direct:
\[
L := \langle T_1, \ldots, T_m \rangle = T_1 \times \cdots \times T_m.
\]
We claim that $L$ contains every automorphic image of $T$. Suppose there exists an image $T^\alpha$ such that $T^\alpha \nsubseteq L$. Since $T^\alpha$ is a minimal normal subgroup, \cref{min_group} implies that 

\[
\langle L, T^\alpha \rangle = L \times T^\alpha = T_1 \times \cdots \times T_m \times T^\alpha,
\]
which contradicts the maximality of $\{T_1, \ldots, T_m\}$. Thus, every automorphic image of $T$ is contained in $L$. Hence, $L$ is invariant under all automorphisms of $F$, so $L$ is a characteristic subgroup. As $F$ is characteristically simple and $L \neq 1$, we must have $L = F$, so $F = T_1 \times \cdots \times T_m$. To show $T$ is simple, let $N \unlhd T$. Since $T$ is a direct factor of $F$, $N$ is normal in $F$. By the minimality of $T$, $N=1$ or $N=T$. Thus $T$ is simple. Since each $T_i \cong T$, we have $F \cong T^m$.
\end{proof}


We will mainly be interested in non-abelian minimal normal subgroups. The following lemma gives a criterion for such a subgroup to be unique. 
\begin{lemma} [Uniqueness of Minimal Normal Subgroups] \label{lem_uniquenss_min}
Let $G$ be a group and $M \trianglelefteq G$ be a  minimal normal subgroup of $G$. Then $M$ is non-abelian and the unique minimal normal subgroup of $G$ if and only if $C_G(M) = 1$.
\end{lemma}
Another equivalent condition can be deduced immediately from the following elementary fact: Let $M$ be a normal subgroup of a group $G$. The conjugation action of $G$ on $M$, denoted by $G \curvearrowright M$, is faithful if and only if $C_G(M) = \ker(G \curvearrowright M) = \{1\}$.

\begin{proof}
We start with \((\implies)\) direction.  Suppose $M$ is the unique minimal normal subgroup of $G$. Since $M \trianglelefteq G$, we know $C_G(M) \trianglelefteq G$. If $C_G(M) \neq 1$, then $C_G(M)$ must contain a minimal normal subgroup $N$ of $G$. By our hypothesis of uniqueness, $N = M$. This implies $M \leq C_G(M)$, which means $M$ is abelian, as elements of $M$ commute with each other. This contradicts the assumption that $M$ is non-abelian. Therefore, $C_G(M) = 1$.
To see the \((\impliedby)\) direction, suppose \(C_G(M)=1\). Let \(N\) be any minimal normal subgroup of \(G\).
By \cref{min_group}, either \(M\cap N=1\) or \(M\leq N\). If
\(M\cap N=1\), then \([M,N]=1\), and hence \(N\leq C_G(M)=1\), a
contradiction. Therefore \(M\leq N\). By the minimality of \(N\), we must have 
$M=N$.
\end{proof}

Let \(N \trianglelefteq M \leq G\). We say that a subgroup \(H \leq G\) \emph{centralizes} the factor \(M/N\) if \([m,h]\in N\) for every \(m\in M\) and \(h\in H\); equivalently, \([M,H]\leq N\). We have the following observations:

\begin{lemma}[Centralization of Factors]
\label{lem:centralization_of_factors}
Let \(K\trianglelefteq M\) with \(K\leq N\) and \(K\leq H\). Then:
\begin{enumerate}[(i)]
\item \(H\) centralizes \(M/N\) if and only if \(H/K\) centralizes \((M/K)/(N/K)\).
\item Under the conjugation action of \(G\) on elements of \(M/N\), \(H\) centralizes \(M/N\) if and only if \(H\leq \ker(G\curvearrowright M/N)\).
\item If \(H\) centralizes \(M/N\), then \((H\cap M)/(H\cap N)\) is abelian.
\end{enumerate}
\end{lemma}

\begin{proof}
For (i), this follows from the correspondence theorem and the identity
\([H/K,M/K]=[H,M]K/K\).

For (ii), this follows immediately from the definitions, since \(H\) centralizes \(M/N\) if and only if every element of \(H\) acts trivially on \(M/N\).

For (iii), since \(H\) centralizes \(M/N\), we have \([M,H]\leq N\). Hence
\([H\cap M,H\cap M]\leq [M,H]\leq N\). Since
\([H\cap M,H\cap M]\leq H\), we obtain
\([H\cap M,H\cap M]\leq H\cap N\). Therefore,
\((H\cap M)/(H\cap N)\) is abelian.
\end{proof}

The following argument is adapted from the proof of \citep[Theorem 1.10]{eberhard2026diameterboundsarbitraryfinite}.

 \begin{prop} \label{bound_on_non_ablina_normal_implies_virtualy_every_finite_quotietn_soluble}
Let $G$ be a finitely generated group. Suppose there exists a constant $C < \infty$ such that every non-abelian minimal normal subgroup of every non-trivial finite quotient of $G$ has order at most $C$. Then $G$ contains a finite-index subgroup $H$ such that every finite quotient of $H$ is soluble.
\end{prop}
Let \(H \leq G\). The \emph{core} of \(H\) in \(G\), denoted by \(\operatorname{core}_G(H)\), is defined by
\[
\operatorname{core}_G(H)=\bigcap_{g\in G} gHg^{-1},
\]
and is the unique largest normal subgroup of \(G\) contained in \(H\).

\begin{proof}
First, note that it is sufficient to find a finite-index subgroup $H \leq G$ such that, for every $N \trianglelefteq_{\mathrm{f.i.}} G$ contained in $H$, the quotient $H/N$ is soluble. Indeed, if $K \trianglelefteq_{\mathrm{f.i.}} H$, then $H/K$ is a quotient of $H/\operatorname{core}_G(K)$, and solubility is preserved under quotients.

By assumption, there exists a constant $C' < \infty$ such that for every $L \trianglelefteq_{\mathrm{f.i.}} G$, every minimal normal subgroup $\mathcal{M}$ of every finite quotient $G/L$ satisfies $\vert{}\operatorname{Aut}(\mathcal{M})\vert{} \leq C'$. Define$$H = \bigcap_{\substack{\mathcal{H} \leq G \\ [G : \mathcal{H}] \leq C'}} \mathcal{H}.$$Since $G$ is finitely generated, it has only finitely many subgroups of index at most $C'$. Hence, $H$ is a finite-index subgroup of $G$. We will show that $H$ satisfies the required property in the proposition.

Let \(N\trianglelefteq_{\mathrm{f.i.}}G\) with \(N\leq H\). We adopt the notation that, for any subgroup \(\mathcal H\leq G\) containing \(N\), \(\overline{\mathcal H}\) denotes the quotient subgroup \(\mathcal H/N\).  Consider a chief series$$\overline{G} = \overline{G}_0 \trianglerighteq \overline{G}_1 \trianglerighteq \dots \trianglerighteq \overline{G}_\ell = 1,$$i.e. each $\overline{G}_i \trianglelefteq \overline{G}$, and the factors $\overline{G}_{i-1}/\overline{G}_i$ are minimal normal subgroups of $\overline{G}/\overline{G}_i$. Note that the second condition is equivalent to saying $G_{i-1}/G_i$ are minimal normal subgroups of the finite quotient $G/G_i$. We show that the induced series
\begin{equation*}
\overline H=\overline G_0\cap\overline H\trianglerighteq
\overline G_1\cap\overline H\trianglerighteq\cdots\trianglerighteq
\overline G_\ell\cap\overline H=1
\end{equation*}
has abelian factors, which implies that \(\overline H\) is soluble. Indeed, $(\overline G_{i-1}\cap\overline H)/(\overline G_i\cap\overline H)$ is isomorphic to a subgroup of $\overline G_{i-1}/\overline G_i$. If $\overline G_{i-1}/\overline G_i$ is abelian, then we are done. Otherwise, $\overline G_{i-1}/\overline G_i$ is non-abelian. By Lemma~\ref{lem:centralization_of_factors}, it is enough to show that $H$ centralises this $G_{i-1}/ G_i$; equivalently, we want to show that $H$ lies in the kernel of the conjugation action of $G$ on $G_{i-1}/ G_i$. This kernel $\ker(G\curvearrowright G_{i-1}/ G_i)$ has index $\vert{}\operatorname{Aut}(G_{i-1}/ G_i)\vert{} \leq C'$ in $G$, and hence the definition of $H$ implies that $H$ is contained in $\ker(G\curvearrowright G_{i-1}/ G_i)$. This concludes the proof.
\end{proof}

We now state an immediate consequence of \citep[Theorem 4.1(3)]{nil}. The key difference is that we require the chosen subgroups to be normal; the trade-off is that the resulting quotient is abelian with a bounded number of generators, rather than cyclic (see also \citep{luca}). We use standard asymptotic notation: $X \ll Y$ (or $Y \gg X$) means $X \leq CY$ for some absolute constant $C > 0$. If the implicit constant depends on some other variable $\alpha$, we indicate this with a subscript; for instance, $X \ll_{\alpha} Y$ means $X \leq C_\alpha Y$ for some constant $C_\alpha > 0$ depending on $\alpha$.

\begin{lemma} \label{k_h_normal}
Let \(F\) be a finite group such that \(\operatorname{diam}(F)\geq |F|^{\alpha}\).
Then there exist normal subgroups \(K, H\trianglelefteq F\) such that
\begin{itemize}
    \item \(H/K\) is abelian;
    \item \(|F:H|\ll_{\alpha}1\);
    \item \(|H:K|\gg_{\alpha}|F|^{\alpha}\);
    \item \(H/K\) is generated by \(d\ll_{\alpha}1\) elements.
\end{itemize}
\end{lemma}

\begin{proof}By \citep[Theorem~4.1(3)]{nil}, there exist subgroups $\widetilde K \trianglelefteq \widetilde H \leq F$ such that$$\vert{}F:\widetilde H\vert{} \ll_{\alpha} 1, \qquad \vert{}\widetilde H:\widetilde K\vert{} \gg_{\alpha} \vert{}F\vert{}^{\alpha},$$and $\widetilde H/\widetilde K$ is cyclic. We take$$H = \operatorname{core}_F(\widetilde H) \quad \text{and} \quad K = \operatorname{core}_F(\widetilde K).$$The  subgroups $H$ and $K$ satisfy the required conditions.\end{proof}

We will also require the following elementary observations to prove a later theorem.

\begin{lemma}\label{lem:intersection-abelian}
Let \(G\) be a group, and let \(H,K,M \leq G\) with \(K \trianglelefteq H\). If \(H/K\) is abelian and \(K \cap M = \{1\}\), then \(H \cap M\) is abelian.
\end{lemma}

\begin{proof}
Since \(H/K\) is abelian, we have \([H,H] \leq K\). Hence
\[
[H \cap M, H \cap M]
\leq [H,H] \cap M
\leq K \cap M
= \{1\},
\]
so \(H \cap M\) is abelian.
\end{proof}

\begin{lem}(\citep[Lemma 8.4]{eberhard2026diameterboundsarbitraryfinite})
\label{lem:monolithic_quotient}
Let $F$ be a finite group with a non-abelian minimal normal subgroup $M$. Then $F$ has a quotient whose unique minimal normal subgroup is isomorphic to $M$.
\end{lem}

\begin{proof}Let $F^* = F/C_F(M)$ and $M^* = MC_F(M)/C_F(M)$. Since $M$ is non-abelian, its center $Z(M)$ is trivial; therefore, $M \cap C_F(M) =Z(M) = 1$. This implies$$M^* = MC_F(M)/C_F(M) \cong M/(M \cap C_F(M)) \cong M.$$Next, we show that $M^*$ is the unique minimal normal subgroup of $F^*$. Consider a non-trivial normal subgroup $L^* \trianglelefteq F^*$ with preimage $L \trianglelefteq F$. As $L^* \neq 1$, we have $L \not\le C_F(M)$, which implies $[L, M] \neq 1$. By \cref{min_group}, $M \leq L$. It follows that $M^* \leq L^*$. Thus, $M^*$ is the unique minimal normal subgroup of $F^*$, and the result holds.\end{proof}

\begin{lemma}\label{lem:order-out-wreath}Let $T$ be a non-abelian simple group and $n \in \mathbb{N}$. Then$$\max \{\, \vert{}g\vert{} \mid g \in \operatorname{Out}(T^n) \,\} \ll \log(\vert{}T\vert{})e^{ n^{0.75}}.$$\end{lemma}

\begin{proof}For finite groups $H$ and $K$, the following bound holds for the maximal order of an element in the wreath product:$$\max\{\vert{}g\vert{} : g \in H \wr K\} \leq \exp(H) \cdot \max\{\vert{}k\vert{} : k \in K\} \leq \vert{}H\vert{} \cdot \max\{\vert{}k\vert{} : k \in K\},$$ where $\exp(H)$ denotes the exponent of $H$. Now, let $T$ be a non-abelian simple group. It is well known that$$\operatorname{Out}(T^n) \cong \operatorname{Out}(T) \wr S_n.$$Hence,$$\max\{\vert{}g\vert{} : g \in \operatorname{Out}(T^n)\} \leq \vert{}\operatorname{Out}(T)\vert{} \cdot \max\{\vert{}\sigma\vert{} : \sigma \in S_n\}.$$By the standard bound (see, for example, \citep{out_small}) for outer automorphism groups of finite non-abelian simple groups,$$\vert{}\operatorname{Out}(T)\vert{} \ll \log \vert{}T\vert{}.$$ Moreover, the maximum order of an element of $S_n$ is given by Landau's function $g(n) = \max\{\vert{}\sigma\vert{} : \sigma \in S_n\}$, which satisfies$$g(n) \leq e^{C\sqrt{n\log n}}$$for some absolute constant $C$ (see \citep{Landau1953}).  The proof follows by noting that $\sqrt{n \log n} \ll n^{0.75}$, which implies $g(n) \ll e^{ n^{0.75}}$. Thus, combining these bounds yields$$\max\{\vert{}g\vert{} : g \in \operatorname{Out}(T^n)\} \ll \log(\vert{}T\vert{})e^{ n^{0.75}}.$$
\end{proof}

\begin{prop}
\label{u_q_all_finite_quotient_bdd}
Suppose $G$ has u.q.($\alpha$). Then there exists a constant $C_\alpha < \infty$ such that every non-abelian minimal normal subgroup of every non-trivial finite quotient of $G$ has order at most $C_\alpha$.
\end{prop}

\begin{proof}
Let $N \trianglelefteq_{\mathrm{f.i.}} G$ with $N \leq H$. We again adopt the notation that, for any subgroup $\mathcal{H} \leq G$ containing $N$, $\overline{\mathcal{H}}$ denotes the quotient subgroup $\mathcal{H}/N$ of $\overline{G} = G/N$.

Let $\overline{M}$ be a minimal normal subgroup of $\overline{G}$. By \cref{k_h_normal}, there exist normal subgroups $K, H \trianglelefteq G$ such that:
\begin{itemize}
    \item $H/K$ is abelian;
    \item $|G:H| \ll_{\alpha} 1$;
    \item $|H:K| \gg_{\alpha} |G|^{\alpha}$;
    \item $H/K$ is generated by $d \ll_{\alpha} 1$ elements.
\end{itemize}

We will split the remaining argument into two cases using \cref{min_group}.
Suppose first that \(\bar K \cap \bar M=\{1\}\). Since \(\bar H/\bar K\) is abelian, \Cref{lem:intersection-abelian} implies that \(\bar H\cap\bar M\) is abelian. Moreover, as \(\bar H\trianglelefteq\bar G\), the subgroup \(\bar H\cap\bar M\) is normal in \(\bar G\). Since \(\bar M\) is a minimal normal subgroup of \(\bar G\), it follows that \(\bar H\cap\bar M\) is either trivial or equal to \(\bar M\). As \(\bar M\) is non-abelian whereas \(\bar H\cap\bar M\) is abelian, we conclude that \(\bar H\cap\bar M=\{1\}\). Also, \(|\bar G:\bar H|\ll_\alpha 1\), so it follows that \(|\bar M:\bar H\cap\bar M| \leq |\bar G:\bar H|\ll_\alpha 1\). Since \(\bar H\cap\bar M=\{1\}\), we must have \(|\bar M|\ll_\alpha 1\), and we are done in this case.

Now suppose that \(\bar K \cap \bar M=\bar M\), which implies that \(\bar K\) contains \(\bar M\). We have that \(\bar H/\bar K\) is generated by \(d\)  elements and is an abelian subquotient of \(\bar G\), with \(|\bar H/\bar K|\gg_\alpha |\bar G|^\alpha>|\bar M|^\alpha\). It follows that there exists an element \(g\in \bar H/\bar K\) with order \(|g|\gg_\alpha |\bar M|^{\alpha/d}\).

We next find an upper bound for $|g|$. According to \cref{lem_uniquenss_min} and \cref{lem:monolithic_quotient}, we may assume $C_{\bar{G}}(\bar{M}) = \{1\}$. Therefore, by considering the conjugation action of $\bar{G}$ on the elements of $\bar{M}$, we may identify $\bar{G}$ as a subgroup of $\operatorname{Aut}(\bar{M})$. According to \cref{lem:char_simple_direct_power}, we may write $\bar{M} = T^n$ for some non-abelian simple group $T$. Furthermore, since $T$ is non-abelian simple, we have $\operatorname{Inn}(\bar{M})= \operatorname{Inn}(T^n) \cong T^n = \bar{M}$. Since  \(\bar K\) contains \(\bar M\), we can deduce that \(\bar H/\bar K\) is in fact a subquotient of \(\operatorname{Out}(\bar M)\). Hence, combining with the lower bound on $|g|$ we found earlier, we have 
\[
|T|^{n\alpha/d} = |\bar M|^{\alpha/d}\ll_\alpha |g| \le\max \{\, \vert{}h\vert{} \mid h \in \operatorname{Out}(T^n) \,\} .
\]
Applying \cref{lem:order-out-wreath}, we get 
\[
|T|^{n\alpha/d}\ll_\alpha \log |T|\, e^{ n^{0.75}}.
\]

We will show that both \(n\) and \(|T|\) are bounded. The left-hand side grows exponentially in \(n\), while the right-hand side is subexponential in \(n\). Therefore, \(n\) must be bounded. Once \(n\) is fixed, the left-hand side is a polynomial function of \(|T|\), whereas the right-hand side is logarithmic in \(|T|\). Hence \(|T|\) must also be bounded. Consequently, we have \(|\bar M|=|T|^n\ll_\alpha 1\). This concludes the proof.
\end{proof}

We now obtain some immediate corollaries of
\cref{bound_on_non_ablina_normal_implies_virtualy_every_finite_quotietn_soluble}
and \cref{u_q_all_finite_quotient_bdd}.

\begin{cor}\label{u_q_every_finite_quotient_is_solbule}
Let \(G\) be a finitely generated group with u.q. Then \(G\) contains a finite-index subgroup \(H\) such that every finite quotient of \(H\) is soluble.
\end{cor}

\begin{rem}
Indeed, by analyzing the non-Frattini chief factors $M/N$ of $G$, where $|G:N| < \infty$ (for background, see \citep{eberhard2025growthresiduallysolublegroups}), Sean Eberhard (private communication) can show that if $G$ satisfies the u.c. property, then $G$ has a finite-index subgroup $H$ such that $H'$ is residually nilpotent. In particular, a finitely generated group with the u.c. property must be virtually (residually nilpotent)-by-abelian.
\end{rem}

\begin{proof}[Proof of \cref{main_thm_residually_soluble}]
Let $G$ be a finitely generated residually soluble group with $u.q.$ \cref{u_q_every_finite_quotient_is_solbule} implies that $G$ contains a finite-index subgroup $H$ such that every finite quotient of $H$ is soluble. Furthermore, $H$ is also residually finite, so every non-trivial element of $H$ survives in some finite quotient of $H$. Such a finite quotient must be soluble by the property of $H$. Hence, $H$ is residually soluble.
\end{proof}
It is clear that residual solubility is not a virtual property, since any nontrivial finite perfect group is virtually residually soluble but not residually soluble. 

To prove our u.c. conjecture for finitely generated linear groups, we require the following lemma, which is based on an answer by Sean Eberhard on Math StackExchange \citep{sean4357453}.
\begin{lem}\label{all_finite_quotient_solbule_of_linear_impileis_soblue}
Let \(G\) be a finitely generated linear group. Then \(G\) is soluble if and only if every finite quotient of \(G\) is soluble.
\end{lem}

\begin{proof}
The forward implication is clear, since quotients of soluble groups are soluble.
Conversely, suppose that \(G\leq \mathrm{GL}_n(K)\) and every finite quotient of
\(G\) is soluble. Let \(f(n)\) be the maximal derived length of a soluble subgroup of
\(\mathrm{GL}_n(F)\), where \(F\) ranges over all fields. By Zassenhaus'
theorem \citep{Zassenhaus1937}, \(f(n)\) is finite for every $n$. We will show that the derived length of G is at most \(f(n)\). Suppose, for a contradiction, that there exists \(g\in G^{(f(n))}\) with \(g\neq 1\).
 Mal'cev showed that there exists
a finite field \(F\) and a homomorphism
\(\pi:G\to \mathrm{GL}_n(F)\) such that \(\pi(g)\neq 1\)
(see \citep{Malcev1940}; see also \citep{Nica2013} for a modern exposition of
this result). Now, since  \(g\in G^{(f(n))}\), we have
\(\pi(g)\in \pi(G^{(f(n))}) = \pi(G)^{(f(n))}\). The group \(\pi(G)\) is a finite quotient of
\(G\), so it is soluble by hypothesis. Since
\(\pi(G)\leq \mathrm{GL}_n(F)\), its derived length is at most \(f(n)\), and
therefore \(\pi(G)^{(f(n))}=1\). This contradicts the fact that
\(\pi(g)\neq 1\). Hence \(G^{(f(n))}=1\), and so \(G\) is soluble.
\end{proof}

\begin{proof}[Proof of \cref{main_thm_linear}]
By \cref{u_q_every_finite_quotient_is_solbule}, \(G\) contains a finite-index
subgroup \(H\) whose finite quotients are all soluble. Since \(H\) is linear,
\cref{all_finite_quotient_solbule_of_linear_impileis_soblue} implies that
\(H\) is soluble. Hence \(G\) is virtually soluble. The statement now follows directly from \cref{main_thm} and \cref{uniform_D_alpha_passes_finte_index_subgroup}.
\end{proof}

\subsection{Some remarks on the u.c.\ conjectures for residually soluble groups}
According to the result established in \cref{main_thm_residually_soluble}, the statement in \cref{u_c_conj} may be reduced to the residually soluble case. We proceed to analyse this specific case below.
\begin{lem} \label{main_thm_implies}
Let $G$ be an infinite, finitely generated, residually finite, and residually soluble group with u.q. Then there exists a finite-index subgroup $\Gamma$ of $G$ and some $L \in \mathbb{N}$ such that:
\begin{enumerate} [(i)]
    \item $\Gamma / \Gamma^{(L)}$ is a torsion-free nilpotent group with Hirsch length $h \geq 1$;
    \item $\Gamma^{(L)}$ does not have a finite-index subgroup that surjects onto $\mathbb{Z}$;
    \item  Setting $K = \Gamma^{(L)}$, we have $[K : K^{(i)}] < \infty$ for all $i \in \mathbb{N}$.
\end{enumerate}
\end{lem}

\begin{proof}
Note that, for every $i$, \( G/G^{(i)} \) is a finitely generated soluble group, hence by \cref{main_thm} is virtually nilpotent. Consider the Hirsch length \( h = h(G/G^{(i)}) \); this is bounded by \citep[Proposition~1.9]{Guo_Tointon}, from which the statement follows immediately.
\end{proof}
\begin{rem}
Referring to the setup in \cref{main_thm_implies}, one way of proving \cref{u_c_conj} is to show that the u.c.\ property actually forces \( G^{(L)} = \{1\} \). If \( G^{(L)} \neq \{1\} \), one would hope to find a sequence of finite quotients that violates \eqref{diam_lower_bound}. We outline some thoughts below: 

\begin{enumerate}[(i)]

\item If one can show that \(K\) is finitely generated, then the fact that
\(\Gamma^{(L)}\) does not virtually surject onto \(\mathbb{Z}\) implies that no
single filtration of \(\Gamma^{(L)}\) with trivial intersection satisfies
\eqref{diam_lower_bound}. A proof of this can be found in
\citep[Theorem~1.2]{Guo_Tointon}.
    
   \item Suppose \(G\) is residually soluble but not soluble. Then one could construct a sequence of finite soluble groups with strictly increasing derived length. It is known that a finite nilpotent group \( G \) can have a subgroup with “small” index and “small” nilpotency step (e.g. consider the dihedral group \( D_{2^k} \)). In contrast, the analogous statement does not hold for finite soluble groups. A well-known theorem by Dixon \citep[Exercise~5.8.7]{DixonMortimer1996} states that every soluble subgroup of \( S_n \) has derived length at most \( b \log n \), where \( b = \tfrac{5}{2\log 3} \). From this, it is immediate that if a finite soluble group \( G \) has derived length \( d \), and \( H \leq G \) has derived length \( d - r \), then \( |G: H| \) is bounded below by an exponential function in \( r \).
\end{enumerate}
\end{rem}

\section{The u.q. property in two classical non-polycyclic soluble groups}

\subsection{Baumslag-Solitar groups}
 We will first consider the Baumslag-Solitar group $G = BS(1, k) = \mathbb{Z}\left[\frac{1}{k}\right] \rtimes_k \langle t \rangle $ where the semidirect product is defined by  $tnt\i = kn$. In fact, \cref{main_thm_bs1_k} follows  from  \citep[Theorem 1.3(3)]{bs_alain_quotient}; For completeness, we provide a self-contained proof here.
We begin with the following lemma about modular arithmetic.
\begin{lem} \label{modular_arith}
Consider $\Z_m$ as a ring. The following statements about an element $a \in \Z_m$ are equivalent:

\begin{enumerate} [(i)]
\item  $a$ is a unit (i.e, $a$ has a multiplicative inverse);
\item $\gcd(a,m) = 1$;
\item $a$ has finite multiplicative order.
\end{enumerate}
\end{lem}

\begin{proof}

We will first prove the implication (i) $\Rightarrow$ (ii). Suppose there exists $c \in \Z$ such that $ac \equiv_m 1$, i.e. $ac-1 = dm$ for some $d \in \Z$. Then $gcd(a,m) | ac-dm = 1$. Therefore $\gcd(a,m)=1$. Next, we will prove the implication (ii) $\Rightarrow$ (i).  Suppose $\gcd(a,m) = 1$, Then by Bézout's identity, there exists $c$,$d \in \Z$ such that $ac+dm = 1$. Therefore $ac \equiv_m 1$, i.e. $c$ is the multiplicative inverse of $a$ in $\Z_m$. The implication (iii) $\Rightarrow (i)$  is clear. Finally, the implication (i) $\Rightarrow$ (iii) follows from the fact that the set of units of $\Z_m$ forms a finite group under multiplication.
\end{proof}

Next, we will describe some normal subgroups of $G = BS(1, k) = \mathbb{Z}\left[\frac{1}{k}\right] \rtimes_k \langle t \rangle$.
\begin{lem}
Let $k \geq 2$ and $m$ be an integer coprime to $k$. Let $n=\ord_m(k)$. Define $H_{m,n} \coloneqq  m\mathbb{Z}\left[\frac{1}{k}\right] \rtimes_k  n \Z  = \{ \left(\frac{am}{k^c}, t^{bn}\right)  \mid    \frac{a}{k^c} \in \mathbb{Z}\left[\frac{1}{k}\right] , b \in  \Z        \}$, this is a normal subgroup of $G$ with index $mn$. 
\end{lem}

\begin{proof}

According to \cref{modular_arith}, $n$ is finite and $H_{m,n}$ is well-defined. Let $\left(\frac{am}{k^c}, t^{bn}\right) \in H_{m,n}$ and $\left(\frac{p}{k^r},t^q\right) \in G$, then  we have
\[\left(\frac{p}{k^r}, t^q\right) \left(\frac{am}{k^c}, t^{bn}\right)  \left(\frac{p}{k^r}, t^q\right) \i = \left(\frac{p}{k^r}  + \frac{am k^q}{k^c} , t^{bn}\right)  - \frac{p}{k^r}  =  \left( (1 - k^{bn}) \frac{p}{k^r} + \frac{am }{k^{c-q}} , t^{bn}\right).   \]
Since $n=\ord_m(k)$, we have $1 - k^{bn}$ being a multiple of $m$. Therefore $H_{m,n}$ is indeed closed under conjugation.

\end{proof}

\begin{proof} [Proof of \cref{main_thm_bs1_k}]
Now, for $n \in \N$,  we let $m= k^n -1$. Let $Q_n = Q_{m,n}= \frac{G}{H_{m,n}} \iso \Z_m \rtimes_{\bar{k}} \Z_n$, and take the generating set $S = S_{m,n} =\{ (\pm1,0),(0, t^{\pm1})  \}$. Let $\bar{S}$ be the image of $S$ in $Q_n$. Then $\bar{S}^n \supseteq \{ (0, i) \mid 0 \leq i \leq n-1 \}$. Hence,$ \bar{S}^{2n+1} \supseteq \{ (k^i, 0)= (0, i) (1, 0 ) (0, -i ) \mid  0 \leq i \leq n-1 \}$, i.e. $\bar{S}^{2n+1}$ covers all the powers of $k$ in the first component.  Recall that the number of digits required to represent an integer $j$ in base $k$ is $\lfloor\log_k j \rfloor + 1$.  It follows that for any $j \in \{0,1,\ldots, m-1 \}$, we can express it as a sum of at most $\lfloor\log_k j \rfloor +1 \leq  \lfloor\log_k m \rfloor +1 = \log_k (m+1)=n $ elements in $\{ k^r \mid  0 \leq r \leq n-1 \}$. Hence, $ \bar{S}^{(2n+1) n} \supseteq \{ (j,0)\mid 0 \leq j \leq m-1 \}$ . Therefore, $\bar{S}^{(2n+1) n+n} \supseteq \{ (j,i)\mid  0 \leq j \leq m-1 , 0 \leq i \leq n-1 \} = Q_n$, and so $\diam_{\bar{S}}(Q_n) \leq (2n+1) n+n = (2n+1) n+n$.  Let $\a \in (0,1]$ and $\eps>0$,  for big enough $n$,  we have 
$$diam_{\bar{S}}(Q_n) \leq (2n+1) n+n \leq 3n^2 < \eps(n (k^n-1))^\a = \eps|Q_n|^\a. $$
It follows that the group $G$ does not have uniformly almost flat quotients.
\begin{rem}
Using a similar construction and logic, the author believes one can show that the group $B_{p,q} = \mathbb{Z}\left[\frac{1}{pq}\right] \rtimes_{p/q} \langle t \rangle$ does not have uniformly almost flat quotients when $p$ and $q$ are distinct primes. The group $B_{p,q}$ appears as the largest residually finite quotient of the Baumslag-Solitar group $BS(p,q) = \langle a, t \mid t a^p t^{-1} = a^q \rangle$.  To validate the failure of the u.c.. property, one may consider the family of finite-index normal subgroups defined by:
\[
H_{m,n} \coloneqq m\mathbb{Z}\left[\frac{1}{pq}\right] \rtimes \langle t^n \rangle,
\]
where we set $n = |p^m - q^m|$. 
\end{rem}

\end{proof}
\subsection{Wreath products}

In this section, we show that the wreath product of a finitely generated abelian group with $\mathbb{Z}$ does not have uniformly almost flat quotients. We begin with a special case:
\begin{prop} \label{prop:wreath_not_uq}
The wreath product $G = \mathbb{Z}_p \wr \mathbb{Z}$ does not have the u.q. property for any prime $p$.
\end{prop}

\begin{proof}
Let $G = \mathbb{Z}_p$ and $S = \{ (\delta_0, 0), (0, 1) \}$ be the standard generating set, where $\delta_0$ is the generator of the $0$-th copy of $\mathbb{Z}_p$. For each $n \in \mathbb{N}$, consider the $n$-periodisation map $\pi_n \colon G \to \mathbb{Z}_p \wr \mathbb{Z}_n$ defined by:
\[ \pi_n(\underline{x}, z) = (\underline{x}^{[n]}, \overline{z}_n), \]
where $\overline{z}_n \in \mathbb{Z}_n$ is the residue of $z$ modulo $n$, and the vector $\underline{x}^{[n]} =( (x^{[n]})_0, \dots, (x^{[n]})_{n-1}) \in (\mathbb{Z}_p)^n$ is given by $(x^{[n]})_i = \sum_{j \equiv i \pmod{n}} x_j$. 

Let $Q_n = \operatorname{Im}(\pi_n) \cong (\mathbb{Z}_p)^n \rtimes \mathbb{Z}_n$. The order of these finite quotients is $|Q_n| = p^n  n$. Let $S_n = \pi_n(S) = \{ (\sigma, \overline{0}_n), (\underline{0}, \overline{1}_n) \}$, where $\sigma = (1, 0, \dots, 0)$. Any element $(\underline{a}, \overline{k}_n) \in Q_n$ can be written as:
\[ (\underline{a}, \overline{k}_n) = \left( \sum_{i=0}^{n-1} a_i \sigma^{(i)}, \overline{0}_n \right) (\underline{0}, \overline{k}_n), \]
where $a_i \in \{0, \dots, p-1\}$ and $\sigma^{(i)}$ denotes the cyclic shift of $\sigma$ by $i$ positions. Since $(\sigma^{(i)}, \overline{0}_n) = (\underline{0}, \overline{1}_n)^i (\sigma, \overline{0}_n) (\underline{0}, \overline{1}_n)^{-i}$, we have:
\[ (\sigma^{(i)}, \overline{0}_n) \in S_n^{2i+1} \subseteq S_n^{2n-1}. \]
Summing these elements to form the vector $\underline{a}$ and then multiplying by the shift $(\underline{0}, \overline{k}_n)$, we find:
\[ (\underline{a}, \overline{k}_n) \in S_n^{(2n-1)np + n} \subseteq S_n^{2pn^2}. \]
This implies that $\operatorname{diam}_{S_n}(Q_n) \le 2n^2p$. Since we also have $|Q_n| = p^n  n$, it follows immediately that $G$ does not have the u.q. property.
\end{proof}

\begin{proof}[Proof of \cref{main_thm__fg_abelian_wreath_not_uq}]
Since $A$ is non-trivial, it admits a surjective homomorphism onto $\mathbb{Z}_p$ for some prime $p$. This induces a surjective homomorphism from $G$ onto $\mathbb{Z}_p \wr \mathbb{Z}$. Since the u.q. property is inherited by the quotient, and $\mathbb{Z}_p \wr \mathbb{Z}$ fails the property by \cref{prop:wreath_not_uq}, $G$ cannot have the u.q. property.
\end{proof}

In fact, the abelian-by-cyclic group $K \rtimes \mathbb{Z}$ does not have the u.c.\ property whenever $K$ has infinite subgroup rank. A proof can be found in \citep[Chapter~4.2]{Guo2025}.

\begin{rem}
It is worth noting that if $A$ contains $\mathbb{Z}$ in \cref{main_thm__fg_abelian_wreath_not_uq}, there are several other ways to see that $G=A \wr \mathbb{Z}$ does not have the u.q.\ property; we mention two here to give an idea. For simplicity, take $A = \mathbb{Z}$. Then we can write $G \cong K \rtimes \mathbb{Z}$, where $K = \bigoplus_{i \in \mathbb{Z}} \mathbb{Z}$ is the base group.
\begin{enumerate}
    \item The group $G$ admits $BS(1,k)=\mathbb{Z}[1/k]\rtimes_k \mathbb{Z}$ as a quotient. Consequently, $G$ does not have uniformly almost flat coset spaces by \cref{main_thm_bs1_k}.
    \item By quotienting out appropriate subgroups of $A$, one can show that $G$ has nilpotent quotients of arbitrarily large Hirsch length. It follows that $G$ does not have uniformly almost flat coset spaces by \citep[Proposition~1.8]{Guo_Tointon}.
\end{enumerate}
\end{rem}

\footnotesize{
\bibliographystyle{abbrv}
\bibliography{Biblio}
}

\end{document}